\documentclass[11pt]{amsart}

\pdfoutput=1

\usepackage[utf8x]{inputenc}
\usepackage[L7x]{fontenc}
\usepackage[pdftex]{graphicx}
\usepackage[T1]{fontenc}
\usepackage{amssymb}
\usepackage{amsmath}
\usepackage{colonequals}
\usepackage{color}
\usepackage{calc}
\usepackage{hyperref}
\usepackage{enumitem}
\usepackage{subfigure}
\usepackage{version}

\includeversion{pictures}
\includeversion{halfplane}

\usepackage{amsthm}
\theoremstyle{definition}
\newtheorem{theorem}{Theorem}
\newtheorem{prop}[theorem]{Proposition}

\newtheorem{lemma}[theorem]{Lemma}
\newtheorem{corollary}[theorem]{Corollary}
\newtheorem{remark}[theorem]{Remark}

\numberwithin{theorem}{section}
\numberwithin{equation}{section}

\newcommand{\bbR}{\mathbb{R}}
\newcommand{\bbZ}{\mathbb{Z}}

\newcommand{\bbC}{\mathbb{C}}
\newcommand{\bbChat}{\hat{\mathbb{C}}}
\newcommand{\bbP}{\mathbb{P}}

\newcommand{\bbH}{\mathbb{H}}
\newcommand{\SLE}{\mbox{SLE}}

\newcommand{\calC}{\mathcal{C}}

\newcommand{\calB}{\mathcal{B}}

\def\Res{\mathop{\rm Res}}
\renewcommand\Im{\mathop{\rm Im}\nolimits}
\renewcommand\Re{\mathop{\rm Re}\nolimits}
\newcommand\proj{\mathop{\rm proj}\nolimits}
 \DeclareMathOperator{\dist}{dist}
\DeclareMathOperator{\id}{id}
\newcommand\pd[1]{\frac{\partial}{\partial #1}}
\newcommand{\eps}{\epsilon}
\newcommand{\struts}[1]{\rule[-#1pt]{0pt}{#1pt*2}}
\hypersetup{pdfauthor={Dana Mendelson, Asaf Nachmias, and Samuel S. Watson},pdftitle={Rate of Convergence for Cardy's Formula}}

\setenumerate{labelsep=3pt,itemsep=0pt,parsep=0pt,topsep=0pt}

\begin{document}
\thispagestyle{empty}

\title[Rate of convergence for Cardy's formula]{Rate of convergence for
  Cardy's formula}
\author[D.~Mendelson, A.~Nachmias, and S.S.\!  Watson]{Dana Mendelson, Asaf
  Nachmias, and Samuel S.\!  Watson}

\maketitle

\thispagestyle{empty}

\begin{abstract}
  We show that crossing probabilities in 2D critical site percolation on
  the triangular lattice in a piecewise analytic Jordan domain converge
  with power law rate in the mesh size to their limit given by the
  Cardy-Smirnov formula. We use this result to obtain new upper and lower
  bounds of $e^{O(\sqrt{\log \log R})} R^{-1/3}$ for the probability that
  the cluster at the origin in the half-plane has diameter $R$, improving
  the previously known estimate of $R^{-1/3+o(1)}$.
\end{abstract}

\section{Introduction}
Let $\Omega\subset \bbC$ be a nonempty Jordan domain, and let $A,B,C,D$ be
four points on $\partial \Omega$ ordered counter-clockwise. Let $P^\delta$
denote the critical site percolation measure on the triangular lattice with
mesh size $\delta>0$, that is, each site in the lattice is independently
declared {\em open} or {\em closed} with probability $1/2$ each.  The
Cardy-Smirnov formula \cite{Smirnov} states that as $\delta \to 0$, the
probability $P^\delta(AB\leftrightarrow CD)$ that there exists a path of
open sites in $\Omega$ starting at the arc $AB$ and ending at the arc $CD$
converges to a limit that is a conformal invariant of the four-pointed
domain (see Figure~\ref{fig:crossprob}).  Our main theorem establishes a
power law rate for this convergence under mild regularity hypotheses.

\begin{theorem}
  \label{thm:cardyrate}
  Let $(\Omega,A,B,C,D)$ be a four-pointed Jordan domain bounded by
  finitely many analytic arcs meeting at positive interior angles.  There
  exists $c>0$ such that
\[
P^\delta(AB\leftrightarrow CD) - \lim_{\delta \to
    0}P^\delta(AB\leftrightarrow CD) = O(\delta^c),
\]
where the implied constants depend only on $(\Omega,A,B,C,D)$.
\end{theorem}

\begin{pictures}
\begin{figure}
  \centering
  \includegraphics{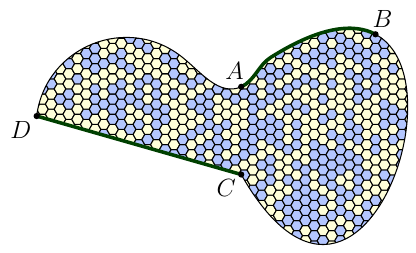}
  \caption{We picture triangular site percolation by coloring the faces of
    the dual hexagonal lattice. Smirnov's theorem states that the probability
    of a yellow crossing from boundary arc $AB$ to boundary arc $CD$
    converges, as the mesh size tends to 0, to a limit which is a conformal
    invariant of the four-pointed domain $(\Omega,A,B,C,D)$.  In the sample
    shown, the yellow crossing event $\{AB\leftrightarrow CD\}$
    occurs.} \label{fig:crossprob}
\end{figure}
\end{pictures}
\noindent We prove Theorem~\ref{thm:cardyrate} for all $c<1/6$, with better exponents
for certain domains (see Remark~\ref{rem:sharp_exponent}).

Schramm posed the problem of improving estimates on percolation arm events
(see Problem 3.1 in \cite{S07}). In Section~\ref{halfarmsection}, we obtain
the following improvement of the estimate found in \cite{SW} for the
probability that the origin is connected to $\{z\,:\,|z|=R\}$ in the upper
half-plane.

\begin{theorem} \label{thm:halfarm} Let $\{0\leftrightarrow S_R\}$ denote
  the event that there exists an open path from the origin to the semicircle
  $S_R$ of radius $R$ in critical site percolation on the triangular
  lattice in the half-plane. Then
  \[ \bbP(0\leftrightarrow S_R) = e^{O(\sqrt{\log\log R})}  R^{-1/3} = (\log
  R)^{O(1/\sqrt{\log \log R})}R^{-1/3}.\]
\end{theorem}

Our methods also yield the estimate $e^{O(\sqrt{\log\log R})}
R^{-1/6\beta}$ for the probability that the origin is connected to
$\{z\,:\,|z|=R\}$ in the sector centered at the origin of angle
$2\pi\beta$. We remark that our methods are insufficient to give better
estimates for the probability that the origin is connected to
$\{z\,:\,|z|=R\}$ in the full plane (the so-called {\em one-arm} exponent,
which takes the value $5/48$, \cite{LSW}) and multiple arm events either in
the full or half plane.

In his proof of Cardy's formula, Smirnov constructs a discrete observable
$G_\delta:\Omega^\delta\to \bbC$, defined as a complex linear combination
of crossing probabilities, and shows that $G_\delta$ converges as
$\delta\to 0$ to a conformal map. The crossing probabilities and their
limits can be then read off $G_\delta$ and its limit. A similar high-level
strategy was also used by Smirnov \cite{Smirnov2} and Chelkak and Smirnov
\cite{CS} to show that the interfaces of the critical Ising and FK-Ising
model converge to SLE curves. See \cite{DCS} for a comprehensive survey of
this subject.

We note that the power law rate of convergence is obtained for the FK-Ising
model (\cite{Smirnov2, HS}) more directly than for percolation, because
the combinatorial relations in the Ising model establish that ``discrete
Cauchy-Riemann'' equations hold precisely. In particular, in the case of
the Ising model one can work with discrete second derivatives and obtain
discrete harmonic functions. By contrast, for percolation the observable
$G_\delta$ is only known to be approximately analytic. Thus it is necessary
to control the global effects of these local deviations from exact
analyticity. To accomplish this, we use a Cauchy integral formula with an
elliptic function kernel in place of the usual $z\mapsto 1/z$.

The half-plane arm exponent, as well as the validity of Smirnov's theorem
is widely believed to be universal in the sense that it should hold for any
reasonable two-dimensional lattice. Nevertheless, so far it is an open
problem to prove Smirnov's theorem even for the case of bond
  percolation on the square lattice.  The value of the exponent does,
however, depend on the dimension. For example, in high dimensions (that is,
dimension at least $19$ in the usual nearest-neighbor lattice, or dimension
at least $6$ on lattices which are spread-out enough) its value is $-3$
\cite{KN}. To the best of our knowledge, there are no predictions in
dimensions $3,4,5$. As for the error terms, in dimension 2 it is believed
that the correct bound for $\bbP(0\leftrightarrow S_R)$ of Theorem
\ref{thm:halfarm} is $\Theta(R^{-1/3})$ (we are unable to prove this
here). In general, it is believed that the polynomial decay should have no
logarithmic corrections
except for at dimension $6$, the upper critical dimension (see \cite{SA}). \\

Finally, we remark that Theorem \ref{thm:cardyrate} has been
independently proved by Binder, Chayes, and Lei \cite{BCL12} using
  different methods.  Their approach applies to arbitrary simply connected
  domains, while our proof achieves explicit exponents for the subclass of
  piecewise analytic domains (see Remark~\ref{rem:sharp_exponent}).

% recently and
% independently obtained by Binder, Chayes and Lei \cite{BCL12} using
% different methods.

\section*{Acknowledgements}
We thank Vincent Beffara, Gady Kozma and
Steffen Rohde, and Scott Sheffield for helpful discussions. We specifically
thank Scott for suggesting the idea to use elliptic functions in the proof
of Theorem~\ref{thm:general} and for his help with the proof of Proposition \ref{prop:cont_exp}.

D.M.~was partially supported by the NSERC Postgraduate Scholarships
Program. A.N.~was supported by NSF grant \#6923910 and NSERC
grant. S.S.W.~was supported by NSF Graduate Research Fellowship Program,
award number 1122374.

\section{Set-up and notation} \label{sec:setup}

Throughout the paper, we consider piecewise analytic Jordan domains
$\Omega$ with positive interior angles. That is, $\partial \Omega$ is a
Jordan curve which can be written as the concatenation of finitely many
analytic arcs $\gamma_1,\ldots,\gamma_N$.  Recall that an arc is said to be
analytic if it can be realized as the image of a closed subinterval
$I\subset \bbR$ under a real-analytic function from $I$ to $\bbC$.  We will
call the point at which two such arcs meet a corner, and we will denote the
collection of corners by $\{x_j\}_{j=1, \ldots, N}$. Our hypotheses imply
that there is a well-defined interior angle at each corner, and we impose
the condition that each such angle lies in $(0,2\pi]$. We define $\tau
\colonequals \exp(2\pi i /3)$ and let $\Omega$ have three marked boundary
points, labeled $x(1)$, $x(\tau)$, and $x(\tau^2)$ in counter-clockwise
order. We denote the angles at marked points by $2\pi\alpha_j$ and those at
unmarked points by $2\pi \beta_i$.

Denote by $\Omega^\delta$ the sites of the triangular lattice with mesh
size $\delta$ which are contained in $\Omega$ or have a neighbor contained
in $\Omega$ and consider critical site percolation on $\Omega^\delta$. Let
$(\Omega^\delta)^*$ be the sites of the hexagonal lattice dual to
$\Omega^\delta$ (that is, $(\Omega^\delta)^*$ are the centers of the
triangles of $\Omega^\delta$). We depict open and closed sites by coloring
the corresponding hexagonal faces yellow and blue, respectively. For
$z,z'\in \partial \Omega$, let $[z,z']$ denote the counter-clockwise
boundary arc from $z$ to $z'$. As in \cite{Smirnov}, the following events
play a central role (see Figure~\ref{defH}):
\[
E^\delta_{\tau^k} (z) =
\left\{
\begin{array}{c}
\exists \text{ a simple open path from } [x(\tau^{k+2}),x(\tau^k)] \text{
  to }[x(\tau^k),x(\tau^{k+1})] \\ \text{ separating }  z \text{ from } [x(\tau^{k+1}),x(\tau^{k+2})]
\end{array}
\right\} \, ,
\]
for $k\in\{0,1,2\}$. Let $H_{\tau^k}^\delta = \bbP(E_{\tau^k}^\delta)$ and
for $z$ and $z+\eta$ neighbors in $(\Omega^\delta)^*$, define
$P^\delta_{\tau^k}(z,\eta)$~$=\bbP(E_{\tau^k}^\delta(z+\eta)\setminus
E_{\tau^k}^\delta(z))$. Following \cite{Beffara}, we define
\[
G^\delta \colonequals H^\delta_1 + \tau H^\delta_\tau + \tau^2 H^\delta_{\tau^2},  \qquad S^\delta \colonequals H^\delta_1 + H^\delta_\tau +
H^\delta_{\tau^2}.
\]

\begin{pictures}
\begin{figure}[h]
\begin{center}
\includegraphics{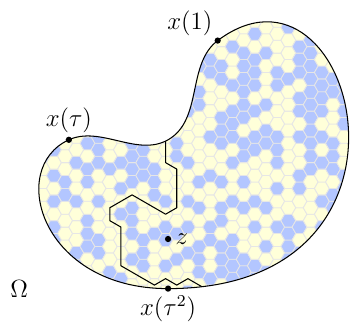}
\caption{The event $E^\delta_{1}(z)$ occurs when there exists a simple open path
  separating $z$ from $[x(\tau),x(\tau^2)]$.}
\label{defH} % for conformal invariance figure
\end{center}
\end{figure}
\end{pictures}

We extend the domain of $G^\delta$ from the lattice $(\Omega^\delta)^*$ to
all of $\Omega$ by triangulating each hexagonal face and linearly
interpolating in each resulting triangle. The possible triangulations
  for each face are \includegraphics[height=10pt]{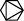}
  and \includegraphics[height=10pt]{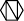} and rotations
  thereof. We will see that the choice of triangulation is immaterial. We
obtain Theorem~\ref{thm:cardyrate} as a corollary of the following theorem.

\begin{theorem} \label{thm:general} Let $(\Omega,x(1),x(\tau),x(\tau^2))$
  be a three-pointed, simply connected Jordan domain bounded by finitely
  many analytic arcs meeting at positive interior angles, and let $T$
    be the triangular domain with vertices $1,\tau,$ and $\tau^2$. Then
  there exists $c>0$ so that $|G^\delta(z) -\phi(z)| =O(\delta^c)$, where
  $\phi$ is the conformal map from $(\Omega,x(1),x(\tau),x(\tau^2))$ to
  $(T,1,\tau,\tau^2)$, and where the implied constants depend only on the
  three-pointed domain.
\end{theorem}

\begin{remark} \label{rem:sharp_exponent} Our methods establish
  Theorem~\ref{thm:general} (and thus Theorem~\ref{thm:cardyrate}) for any
  exponent
  \begin{equation} \label{eq:bounds}
  c<\min_{i,j}\left(\frac{2}{3},\frac{1}{6\alpha_i},\frac{1}{2\beta_j}\right).
  \end{equation}
  These exponents are essentially the best possible given our approach,
  because no piecewise-linear interpolant of a function on a lattice of
  mesh $\delta$ can approximate the conformal map to $T$ with error better
  than $\delta^{\min_{i,j}(1/6\alpha_i,1/2\beta_j)}$ due to behavior near
  the boundary. %  It is conceivable that the right-hand side of
  % \eqref{eq:bounds} could be increased, but it would imply very strong
  % estimates. For example, an exponent better than 1/3 for the half-disk
  % would establish convergence of $R^{1/3}\bbP(0\leftrightarrow S_R)$ with
  % power law rate, a substantial improvement over Theorem~\ref{thm:halfarm}
  % and the conjectured $\bbP(0\leftrightarrow S_R)=\Theta(R^{-1/3})$.
\end{remark}

\begin{remark}
  Our proof of Theorem~\ref{thm:cardyrate} uses results whose proofs
  require SLE tools, but only for two purposes: (1) to handle the case
  where the domain contains reflex angles (that is, some interior angle
  formed at the intersection of two of the bounding analytic arcs is
  greater than $\pi$), and (2) to obtain the sharp exponent discussed in
  Remark~\ref{rem:sharp_exponent}. Without SLE machinery, we obtain
  Theorem~\ref{thm:cardyrate} for domains without reflex angles and for
  exponents $c<\min_{i,j}(c_3,1/6\alpha_i,1/6\beta_j)$, where $c_3$ is the
  three-arm whole-plane exponent (which is known to be 2/3, but only by
  using an $\SLE$ convergence result). See Remark~\ref{rem:sle_free} for
  further discussion of this point.
\end{remark}

\begin{remark}
  In \cite{SW}, a bound of $R^{-1/3+o(1)}$ for the half-plane arm exponent
  was proved using SLE calculations and the fact that the percolation
  exploration path converges to SLE$_6$ as proved by Smirnov \cite{Smirnov}
  and Camia-Newman \cite{CN}. By contrast, our proof follows from
  Proposition~\ref{prop:uniform_constant}, which is a variation of
  Theorem~\ref{thm:cardyrate} proved by similar methods. The only SLE
  result on which our proof of Theorem~\ref{thm:halfarm} depends is the
  statement $c_3>1/3$, where $c_3$ is the three-arm whole-plane exponent.
\end{remark}

For two quantities $f(\delta)$ and $g(\delta)$, we use the usual asymptotic
notation $f=O(g)$ to mean that there exist constants $C$ and $\delta_0>0$
so that $|f(\delta)|\leq C|g(\delta)|$ for all $0<\delta<\delta_0$. We use
the notation $f\lesssim g$ to mean $f=O(g)$ as $\delta \to 0$, and we write
$f\asymp g$ to mean $f=O(g)$ and $g=O(f)$. We sometimes use $C$ to denote
an arbitrary constant.

\section{Preliminaries} \label{sect:prelim}

First we recall some results from \cite{Smirnov}. The first is a H\"older norm estimate of $H_{\tau^k}$ and is obtained via Russo-Seymour-Welsh estimates.

\begin{lemma}[Lemma 2.2 in \cite{Smirnov}] \label{lem:Holder}
  There exist $C,c>0$ depending only on $\Omega$ such that for all
  $\delta>0$, the $c$-H\"older norm of $H_{\tau^k}^{\delta}$ is bounded above by
  $C$. That is,
\begin{equation} \label{HolderProp}
|H^{\delta}_{\tau^k}(z)-H^{\delta}_{\tau^k}(z')| \leq C |z-z'|^c,
\end{equation}
for ${\tau^k} \in \{1,\tau,\tau^2\}$.
\end{lemma}

\begin{pictures}
\begin{figure}
  \centering
  \includegraphics{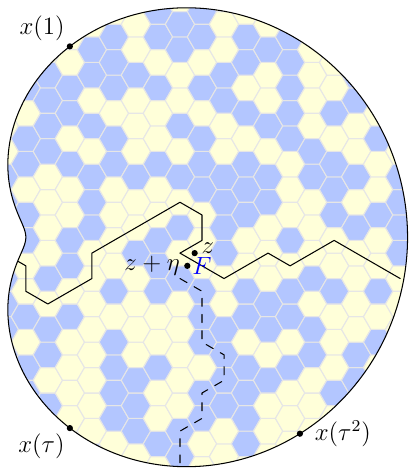}
  \caption{The event $E_1(z)\setminus E_1(z+\eta)$ occurs if and only if
    there are disjoint yellow arms from $z$ to $[x(\tau^2),x(1)]$ and from $z$ to
    $[x(1),x(\tau)]$ forming a simple path separating $z$ from
    $[x(\tau),x(\tau^2)]$, as well as a blue arm from $z+\eta$ to
    $[x(\tau),x(\tau^2)]$ which prevents a yellow path from separating
    $z+\eta$ as well.} \label{fig:threearm}
\end{figure}
\end{pictures}

Our second estimate is Smirnov's ``color switching'' lemma.

\begin{prop} [Lemma 2.1 in \cite{Smirnov}] \label{prop:colorswitching} For every vertex $z\in
  (\Omega^\delta)^*$ and $k\in \{0,1,2\}$, we have
  \[P^{\delta}_{\tau^k}(z,\eta) = P^{\delta}_{\tau^{k+1} }(z,\tau\eta).\]
\end{prop}

We will sometimes drop the superscript $\delta$ from the notation
  when it's clear from context. If $F$ is a hexagonal face in
$(\Omega^\delta)^*$, let $V(F)$ denote the set of vertices of $F$ and
define for each $z\in V(F)$ the vector $\eta$ pointing to the adjacent
vertex counterclockwise from $z$. Define the difference (see Figure~\ref{fig:difference_events}(a))
\[
R_k(z)\colonequals
|P_{\tau^k}(z+\tau^k
\eta,-\tau^k\eta)-P_{\tau^k}(z+\tau^{k+1}\eta,-\tau^k\eta)|.
\] 
Define $z'=z+\tau\eta-\eta$
and rewrite $P_{\tau^k}(z',\eta)$ as $P^{\Omega'}_{\tau^k}(z,\eta)$, where
$\Omega'$ is obtained by translating $\Omega$ by $z-z'$ (and $P^{\Omega'}$
refers to probability with respect to $\Omega'$). Define the events
$E_{\tau^k}'(z)$ with respect to $\Omega'$, and define $x'(\tau^k)$ to be
$x(\tau^k)$ translated by $z-z'$.

Given $k,l\in\{0,1,2\}$, $\sigma \in \{-1,1\}$, and $z\in
(\Omega^\delta)^*$, we say that the event $E^\text{five
  arm}_{\tau^k,\tau^l,\sigma}(z)$ occurs if
\begin{itemize}
\item $\sigma = 1$, and $E_{\tau^k}(z)\setminus
E_{\tau^k}(z+\tau^k\eta)$ occurs, and the arm from $z$ to \\
$[x(\tau^{l+1}),x(\tau^{l+2})]$ fails to connect in $\Omega'$, or
\item $\sigma = -1$, and $E'_{\tau^k}(z)\setminus
E'_{\tau^k}(z+\tau^k\eta)$ occurs, and the arm from $z$ to \\
$[x'(\tau^{l+1}),x'(\tau^{l+2})]$ fails to connect in $\Omega$.
\end{itemize}
For $z_0\in \Omega$, we define $E^\text{five
    arm}_{\tau^k,\tau^l,\sigma}(z)$ to be the union of $E^\text{five
    arm}_{\tau^k,\tau^l,\sigma}(z)$ as $z$ ranges over the vertices of the
  hexagonal face containing $z_0$.

Note that these are indeed five-arm events because two additional arms are
required to prevent the failed arm from connecting elsewhere on
$[x(\tau^l),x(\tau^{l+1})]$ (see Figure~\ref{fig:fivearm}).

\begin{prop} \label{prop:CR} If $F$ is a hexagonal face in
  $(\Omega^\delta)^*$, then for $z_0$ in the interior of $F$ we have
  \begin{align} \label{CReq} \delta |\bar\partial G^\delta(z_0)|&\leq
    3\sqrt{3} \max_{z \in V(F),\,k\in\{0,1,2\} } R_k(z) \\ \label{CReq2}
    &\leq 54\sqrt{3} \max_{k,l\in
      \{0,1,2\},\sigma\in
      \{-1,1\}}\bbP(E_{\tau^k,\tau^l,\sigma}^{\text{five arm}}(z_0)).
\end{align}
\end{prop}

\begin{pictures}
\begin{figure}
  \centering
  \includegraphics[width=8cm]{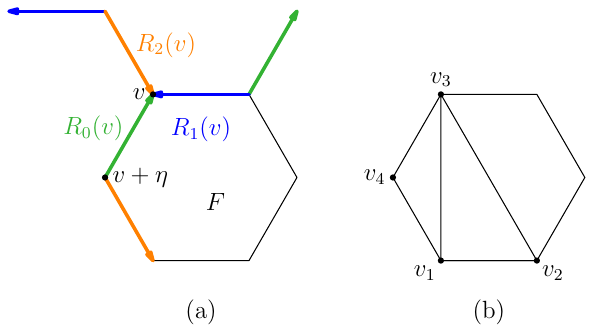}
  \caption{(a) Each arrow represents the probability of a three-arm event
    as shown in Figure~\ref{fig:threearm}. The quantity $R_0(z)$ is defined
    to be the difference between the probabilities represented by the two
    green arrows. Similarly, $R_1(z)$ is shown in blue and $R_2(z)$ is
    shown in orange. (b) Suppose that the triangle $z_1z_2z_3$ is in the
    triangulation of the face $F$. For $z$ in the interior of this
    triangle, we bound $\bar\partial G^\delta(z)$ by applying
    \eqref{eq:adj} to triangles $z_2z_1z_4$ and $z_1z_4z_3$.
  } \label{fig:difference_events}
\end{figure}
\end{pictures}

\begin{pictures}
\begin{figure}
  \centering
  \includegraphics{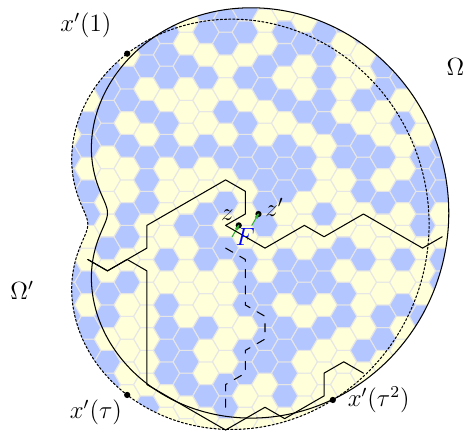}
  \caption{ The symmetric difference of the events $E_{1}(z)\setminus
    E_{1}(z+\eta)$ and $E'_{1}(z')\setminus E'_{1}(z'+\eta)$
    can occur in six ways. One way for the event to occur is shown above:
    the three requisite arms are present in $\Omega$, so the event
    $E_{1}(z)\setminus E_{1}(z+\eta)$ occurs. However, the blue
    arm fails to connect to $[x'(\tau),x'(\tau^2)]$ in $\Omega'$. This
    requires two additional yellow arms to prevent the blue arm from
    connecting elsewhere on $[x'(\tau),x'(\tau^2)]$. This event is denoted
    $E^\text{five arm}_{1,1,1}(z)$. The first subscript $\tau^k$ specifies
    that the three-arm event under consideration involves the blue arm
    touching down on $[x(\tau^{k+1}),x(\tau^{k+2})]$. The second subscript
    $\tau^l$ indicates that the boundary arc
    $[x(\tau^{l+1}),x(\tau^{l+2})]$ is involved in a failed connection.
    The third subscript $\sigma$ describes whether the failed connection
    occurs in $\Omega$ but not $\Omega'$ (in which case we say $\sigma =
    1$), or vice versa ($\sigma = -1$).} \label{fig:fivearm}
\end{figure}
\end{pictures}

\begin{proof}
The main idea in the following proof is suggested in \cite{Smirnov}. For
\eqref{CReq}, we first observe that for $z\in V(F)$, we have
\begin{align*}
  \delta \left[ \pd{\eta}H_{\tau^k}(z) -\pd{(\tau\eta)}H_{\tau^{k+1}}(z)
  \right]&= P_{\tau^k}(z,\eta) - P_{\tau^k}(z+\eta,-\eta)\\ &\hspace{2cm}-
  P_{\tau^{k+1}}(z,\tau\eta)
  + P_{\tau^{k+1}}(z+\tau\eta,-\tau\eta)  \\
  &= P_{\tau^{k+1}}(z+\tau\eta,-\tau\eta)-P_{\tau^k}(z+\eta,-\eta)
  \\ &= P_{ {\tau^k}}(z+\tau\eta,-\eta)-P_{\tau^k}(z+\eta,-\eta),
\end{align*}
by Proposition~\ref{prop:colorswitching}.  Suppose that the triangle $T$
with vertices $z$, $z+\eta$, and $z+\tau\eta$ is in the triangulation of
$F$. Then for $z$ in the interior of $T$, we may write $\delta \bar \partial$ as
$\delta\lambda\left(\frac{\partial}{\partial \eta} -
  \frac1\tau\frac{\partial}{\partial (\tau\eta)}\right)$, where $\lambda =
1/2+i/(2\sqrt{3})$. We obtain
\begin{align} \nonumber
  \delta \left|\lambda  \left(\pd{\eta} - \frac{1}{\tau}\right.\right.&\left.\left.\pd{(\tau\eta)}\right)\vphantom{\pd{\eta} }(H_1+\tau
    H_\tau + \tau^2 H_{\tau^2})\right| \\ \nonumber
  &= |\lambda|\left|\left( \frac{\partial H_1}{\partial
        \eta}-\frac{\partial H_\tau}{\partial \tau\eta}\right) + \tau
    \left( \frac{\partial H_\tau}{\partial \tau \eta}-\frac{\partial
        H_{\tau^2}}{\partial \tau^2\eta}\right) + \tau^2 \left(
      \frac{\partial H_\tau^2}{\partial \tau^2\eta}-\frac{\partial
        H_1}{\partial \eta}\right)\right| \\ \label{eq:adj}
  &\leq \sqrt{3} \max_{k\in \{0,1,2\}}\left|P_{\tau^k}(z+\tau^k
\eta,-\tau^k\eta)-P_{\tau^k}(z+\tau^{k+1}\eta,-\tau^k\eta)\right|.
\end{align}
For triangles whose vertices are not consecutive vertices of the hexagon,
we obtain a similar bound by applying \eqref{eq:adj} two or three
times (see Figure~\ref{fig:difference_events}(b)).
% DON'T DELETE THE FOLLOWING COMMENTED SECTION. IT CONTAINS A
% POSSIBLY VALUABLE CALCULATION THAT NONETHELESS IS TOO SIMPLE
% TO STAY IN THE PAPER
%For example, consider
%the right triangle with vertices $v_1=v$, $v_2=v+\eta$, and
%$v_3=v+\tau\eta-\tau^2\eta$. If this triangle is part of the triangulation
%of $F$, then we may write $\bar\partial=\frac{\delta}{2}\left(\frac{\partial}{\partial
%\eta} - \frac{1}{i\sqrt{3}} \frac{\partial}{\partial(i\sqrt{3}\eta)}\right)$ and apply
%\eqref{eq:adj} twice to obtain
%\begin{align*}
%\delta\bar\partial G^\delta(z) &=
%\frac{1}{2}\left[G^\delta(v_2)-G^\delta(v_1) -
%  \frac{1}{i\sqrt{3}}(G^\delta(v_3)-G^\delta(v_1))\right] \\
%&= \frac{1}{2}\left[G^\delta(v_2)-G^\delta(v_1) -
%  \frac{1}{i\sqrt{3}}(G^\delta(v_3)-G^\delta(v_4) + G^\delta(v_4) -
%  G^\delta(v_1))\right] \\
%&= \frac{1}{2}\left[G^\delta(v_2)-G^\delta(v_1) -
%  \frac{1}{i\sqrt{3}}(\tau(G^\delta(v_1)-G^\delta(v_4)) + G^\delta(v_4) -
%  G^\delta(v_1))+E\right] \\
%&= \frac{1}{2}\left[G^\delta(v_2)-G^\delta(v_1) -
%  \frac{1}{i\sqrt{3}}(\tau(1-\tau)(G^\delta(v_2)
%  -G^\delta(v_1))+E+E'\right] \\
%&= E+E',
%\end{align*}
%where $v_4=v+\tau\eta$, and the errors $E$ and $E'$ are bounded in absolute
%value by $\max_{v \in V(F),\,k\in\{0,1,2\} } R_k(v)$. A similar argument
%works for any triangle with vertices in $V(F)$. Thus we obtain
%\eqref{CReq} for all $z$ in the interiors of the triangles in the
%triangulation of $F$.

For the bound in \eqref{CReq2}, we let $A= E_{\tau^k}(z)\setminus
E_{\tau^k}(z+\tau^k\eta)$ and $B=E'_{\tau^k}(z) \setminus
E'_{\tau^k}(z+\tau^k\eta)$ and apply $ |\bbP(A) - \bbP(B)|\leq
\bbP(A\bigtriangleup B),$ where $A \bigtriangleup B$ denotes the symmetric
difference of $A$ and $B$.  Note that $A \bigtriangleup B \subset
\bigcup_{k,l,\sigma}E^{\text{five
    arm}}_{\tau^k,\tau^l,\sigma}(z_0)$, since some arm in $\Omega$
must fail to connect in $\Omega'$, or vice versa. Applying a union bound as
$k$ and $l$ range over $\{0,1,2\}$ and $\sigma$ ranges over $\{-1,1\}$
yields the result.
\end{proof}

Finally, we need the following a priori estimates for $H_{\tau^k}(z)$ when $z$
is near $\partial \Omega$.

\begin{prop} \label{prop:boundary} Let $(\Omega,x(1),x(\tau),x(\tau^2))$ be
  a three-pointed Jordan domain.  There exists $c>0$ such that for every $z
  \in (\Omega_\delta)^*$ which is closer to $[x(\tau^{k+1}),x(\tau^{k+2})]$
  than to $\partial \Omega \setminus [x(\tau^{k+1}),x(\tau^{k+2})]$, the
  following statements hold.
\begin{enumerate}[label=(\roman*)]
\item $H_{\tau^k}(z)\lesssim \dist(z,\partial \Omega)^{c}  $.
\item $|S(z) - 1 |\lesssim \dist(z,\partial\Omega)^c$.
\item $\dist(G^\delta(z),[x(\tau^{k+1}),x(\tau^{k+2})]) \lesssim \dist(z,[\tau^{k+1},\tau^{k+2}])^{c}$,
\end{enumerate}
with implied constants depending only on
$(\Omega,x(1),x(\tau),x(\tau^2))$.
\end{prop}

\begin{proof}
  (i) For $w\in [x(\tau^{k+1}),x(\tau^{k+2})]$, define $D_1(w)$ and
  $D_2(w)$ to be the distances from $w$ to the boundary arcs
  $[x(\tau^{k+2}),x(\tau^k)]$ and $[x(\tau^{k}),x(\tau^{k+1})]$,
  respectively. Let $D=\inf_{w\in
    [x(\tau^{k+1}),x(\tau^{k+2})]}\max(D_1(w),D_2(w))>0$. Let $z'\in
  [x(\tau^{k+1}),x(\tau^{k+2})]$ be a closest point to $z$, and consider
  the annulus centered at $z'$ with inner radius $|z - z'|$ and outer
  radius $R$. Then $E_{\tau^k}(z)$ entails a crossing of this annulus,
  which has probability $O(|z-z'|^c)$ by Russo-Seymour-Welsh.

  (ii) Again let $z'\in [x(\tau^{k+1}),x(\tau^{k+2})]$ be a point nearest to $z$. Consider the
  event that there is a yellow crossing from $[x(\tau^{k+2}),x(\tau^k)]$ to
  $[x(\tau^{k+1}),z']$ and the event that there is a blue crossing from
  $[x(\tau^{k}),x(\tau^{k+1})]$ to $[z',x(\tau^{k+2})]$. These events are mutually
  exclusive, and their union has probability 1. Since these two
  events have probability $H_{\tau^{k+1}}(z)$ and $H_{\tau^{k+2}}(z)$, we see that
  \begin{align*}
  H_{\tau^k}(z) + (H_{\tau^{k+1}}(z) + H_{\tau^{k+2}}(z)) &= O((\dist(z,\partial \Omega)^c) +
  1.
  \end{align*}

  (iii) This statement says that $G$ maps points near each boundary arc to
  the corresponding image segment in the triangle, and it follows
  directly from (i).
\end{proof}

\subsection{Percolation Estimates}

In this subsection we present several percolation-related estimates in
preparation for the proof of Theorem~\ref{thm:general}. We think of these
lattices as embedded in $\bbR^2$ with mesh size $\delta$, and distances are
measured in the Euclidean metric.

Define $\mathcal{C}^k_{\theta}(r,R)$ to be the event that there exist $k$
disjoint crossings of alternating colors from the inner to the outer
boundary of an annular section $A_\theta(r,R)$ of angle $\theta$ and inner
radius $r$ and outer radius $R$. The following is a well-known result on
the half-annulus two-arm and three-arm exponents. We refer the reader to
\cite[Appendix A]{LSW} for a proof.

\begin{prop} \label{prop:twothreearm} We have
\begin{align*}
  P^\delta(\calC^2_\pi(r,R)) &\asymp \frac{r}{R}, \text{ and}\\
  P^\delta(\calC^3_\pi(r,R)) &\asymp \left(\frac{r}{R}\right)^2 .
\end{align*}
\end{prop}
In the next proposition, we show that the exponents in the estimates above
are continuous in the angle $\theta$.

\begin{prop} \label{prop:cont_exp} % There exists a constant $C>0$ such that
  % for all $\eps>0$ there exist $\alpha=\alpha(\eps)>0$ and
  % $\delta_0=\delta_0(\eps)>0$ such that for all $0 < r < R$ and for all
  % $0<\delta<\delta_0$, we have
  For all $\eps>0$, there exists $\alpha=\alpha(\eps)>0$ so that
\begin{align}
P^\delta(\calC^2_{\pi+\alpha}(r,R)) &\lesssim
\left(\frac{r}{R}\right)^{1-\eps}, \text{ and} \label{eqn:two} \\
P^\delta(\calC^3_{\pi+\alpha}(r,R)) &\lesssim
\left(\frac{r}{R}\right)^{2-\eps}, \label{eqn:three}
\end{align}
with implied constants depending only on $\eps$.
\end{prop}

\begin{proof} We only prove (\ref{eqn:two}) since the proof of
  (\ref{eqn:three}) is essentially the same. We begin by showing that there
  exists $C>0$ so that for all $r>0$ and $R>0$, there exists
  $\delta_0=\delta_0(r,R,\eps)>0$ for which
  $P^\delta(\calC^2_{\pi+\alpha}(r,R)) \leq C
  \left(\frac{r}{R}\right)^{1-\eps}$ holds when $0<\delta<\delta_0$. For
  this statement, we may assume without loss of generality that $R=1$.

  Consider the sector of angle $\pi + \alpha$ as a union of a sector of
  angle $\pi$ with a sector of angle $\alpha$. Divide the sector of angle
  $\alpha$ into $\lceil (1-r)\alpha^{-1} \rceil $ curvilinear
  quadrilaterals of radial dimension $\alpha$, as shown in
  Figure~\ref{fig:sector}. Let $s \in \{1, \ldots,\lceil
  (1-r)\alpha^{-1}\rceil \}$ and note that the event
  $\mathcal{C}^2_{\pi+\alpha}(r,1) \setminus \mathcal{C}^2_{\pi}(r,1)$
  entails the existence of a quadrilateral of distance $s\alpha$ from the
  inner circle of radius $r$ such that there is a three-arm crossing of
  alternating colors of the half-annulus with inner radius $\alpha$ and
  outer radius $s\alpha \wedge (1-r-(s+1)\alpha)$.

  In the case $s\alpha \leq (1-r)/2$, there is also a two-arm crossing from
  the annulus of inner radius $s\alpha$ and outer radius $s\alpha + r$ (see
  Figure~\ref{fig:sector}(a)). If $s\in [2^k,2^{k+1}]$ and $s\alpha \leq
  (1-r)/2$, then the probability that both of these events occur is
  $O\left(\left(\frac{\alpha}{s\alpha}\right)^2\left(\frac{\alpha}{\alpha
          s +r}\right)\right)=O(\frac{\alpha}{2^k})$ by Proposition
  \ref{prop:twothreearm}. Applying a union bound over $s$ we obtain
\begin{equation} \label{diffevent}
P^\delta(\mathcal{C}^2_{\pi+\alpha}(r,1)\setminus
  \mathcal{C}^2_{\pi}(r,1)) \leq c \alpha \log^{-1} \alpha.
\end{equation}
Since $\mathcal{C}^2_{\pi}(r,1)\subset \mathcal{C}^2_{\pi+\alpha}(r,1)$,
(\ref{diffevent}) implies
\begin{align*}
P^\delta(\mathcal{C}^2_{\pi+\alpha}(r,1)) &\leq P^\delta(\mathcal{C}^2_{\pi}(r,1) + c
\alpha \log \alpha^{-1} \\
&\leq c(r + \alpha \log \alpha^{-1}).
\end{align*}

In the case $s\alpha > (1-r)/2$, the event $\mathcal{C}^2_{\pi+\alpha}(r,1) \setminus
  \mathcal{C}^2_{\pi}(r,1)$ implies the
existence of a two-arm crossing of alternating colors from the annulus of
inner radius $s\alpha$ and outer radius $s\alpha - r$ and a similar
computation yields $P^\delta(\mathcal{C}^2_{\pi+\alpha}(r,1)) \leq c(r +
\alpha \log \alpha^{-1})$ in this case as well.

\begin{pictures}
\begin{figure}[h]
\begin{center}
\includegraphics{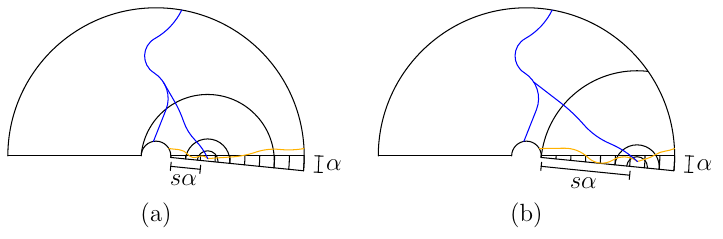}
\begin{minipage}{\linewidth}
  \caption{The cases (a) $s\alpha \leq (1-r)/2$ and (b) $s\alpha > (1-r)/2$
    for the event $\mathcal{C}_{\pi+\alpha}^2(r,R)\setminus
    \mathcal{C}_\pi^2(r,R)$ in
    Proposition~\ref{prop:cont_exp}.}  \label{fig:sector}
\end{minipage}
\end{center}
\end{figure}
\end{pictures}

Finally, to show that $\delta_0$ may be taken to be independent of $r$ and $R$, we
apply a multiplicative argument. Let $K>0$ be large enough and $\delta_0$
small enough that $P^\delta(\mathcal{C}_\pi^2(r,R))<(1/K)^{1-\eps}$
for all $0<\delta<\delta_0$. Insert concentric arcs of radii $r, rK, rK^2,
\ldots, r K^{\lfloor \log_K(R/r)\rfloor }$ between the arcs of radii $r$ and $R$,
and consider the regions between successive pairs of these arcs. Since a
crossing from the arcs of radius $r$ to the arc of radius $R$ implies that
each of these regions is crossed, we have
\begin{align*}
  P^\delta(\mathcal{C}^2_{\pi+\alpha}(r,R)) &\leq
  \prod_{k=1}^{\lfloor \log_K(R/r)\rfloor } P^\delta\left(\mathcal{C}_{\pi+\alpha}^2\left(rK^k,rK^{k+1}\right)\right)\\
  &\leq C \left(\frac{r}{R}\right)^{1-\eps}. \qedhere
\end{align*}
\end{proof}

\begin{remark}
  In particular, by taking $r = \delta$, the previous results yields bounds
  for half-disk crossing probabilities for $z \in \partial \Omega$.
\end{remark}

Using Smirnov's theorem, we can generalize one-arm estimates to annulus
sectors of any angle.

\begin{prop} \label{prop:wedgeonearm}
For every $\eps>0$,
\begin{equation} \label{wedgeonearmeq}
P^\delta(\mathcal{C}_{\theta}^1(r,R)) \lesssim
\left(\frac{r}{R}\right)^{\frac{1}{3\theta}-\eps}.
\end{equation}
\end{prop}

\begin{proof}
  Smirnov's theorem implies that for all $r$ and $R$ there exists
  $\delta_0=\delta_0(r,R,\eps)>0$ so that for all $0<\delta<\delta_0$, we
  have $P^\delta(\mathcal{C}_\theta^1(r,R))\leq
  (r/R)^{1/3\theta-\eps}$. As in the previous proposition, we can
  remove the dependence on $r$ and $R$ with a multiplicative argument.
% LEAVE THIS CALCULATION IN THE PAPER BUT COMMENTED OUT
%  Choose $K$ large enough and $\delta_0$ small enough that
%  $P^\delta(\mathcal{C}_\pi^1(1,K)) \leq (1/K)^{1/3\theta-\eps}$ for all
%  $0<\delta<\delta_0$. Dividing the sector into regions using the arcs
%  of radii $r, rK, rK^2,\ldots rK^{\lfloor \log_K(R/r)\rfloor }$, we use
%  independence to calculate
% \begin{equation} \label{wedgeonearmeq2}
% P^\delta(C_\theta^1(r,R)) \leq
% (K^{1/3\theta-\eps})^{\log(R/r)/\log(1/\eta)-1} \leq C(r/R)^{1/3\theta - \eps}.
% \qedhere
% \end{equation}
\end{proof}

We can generalize the previous results for annular regions to a
neighborhood of a meeting point of two analytic arcs. We let
$\mathcal{C}_{\Omega,z}^k(r,R)$ denote the event that there exist $k$
disjoint crossings of alternating color contained in $\Omega$ and
connecting the circles of radius $r$ and $R$ centered at $z$. We have the
following corollary of Propositions~\ref{prop:cont_exp} and
\ref{prop:wedgeonearm}.

\begin{corollary}\label{cor:smoothtwothree} Let $\eps>0$, let
  $\alpha=\alpha(\eps)$ be an angle satisfying the conclusion in
  Proposition~\ref{prop:cont_exp}. Let $\Omega$ be a piecewise analytic
  Jordan domain in $\bbR^2$. Fix $z \in \partial \Omega$ and suppose that
  $z$ is not a corner of $\Omega$. Let $R_0=R_0(z,\eps) > 0$ be
  sufficiently small that $B_{R_0}(z)\cap \Omega$ is contained in a sector
  centered at $z$ and having angle $\pi+\alpha$ and radius $R_0$. Then for
  all $k\in \{1,2,3\}$ and for all $0 < r < R \leq R_0$,
\begin{equation} \label{eqn:domaincrossing}
P^\delta(\mathcal{C}^k_{\Omega,z}(r,R)) \lesssim \left(\frac{r}{R} \right)^{k(k+1)/6 - \eps},
\end{equation}
with implied constants depending only on $\eps$.
\end{corollary}
\begin{proof}
Since the event $\mathcal{C}^k_\Omega(r,R)$ implies a crossing of a sector
of angle $\pi+\alpha$ with inner and outer radii of $r$ and $R$,
\[
P^\delta(\mathcal{C}^k_\Omega(r,R)) \leq P^\delta(\mathcal{C}_{\pi + \alpha}^k(r,R))
\]
and we can estimate the probability on the right by
Proposition~\ref{prop:cont_exp} for $k\in \{2,3\}$ or
Proposition~\ref{prop:wedgeonearm} for $k=1$.
\end{proof}

We conclude this section by recording a generalization of the previous
corollary for corners $z\in \partial \Omega$. The proof of this proposition
uses convergence of the exploration path to $\SLE_6$. We know how to remove
this dependence on SLE results only when $k=1$, where Smirnov's theorem
suffices.  We use \eqref{eqn:wedgemultiarmeq} when $k\in \{2,3\}$ only to
handle the case where $\Omega$ has reflex angles and to obtain the sharp
exponent discussed in Remark~\ref{rem:sharp_exponent}.

\begin{prop} \label{cor:smoothwedgemultiarm} Suppose that $z\in \partial
  \Omega$ is a corner of $\Omega$, but otherwise the hypotheses and
  variable definitions are the same as in Corollary~\ref{cor:smoothtwothree}.
  Then the conclusion holds, with \eqref{eqn:domaincrossing} replaced by
\begin{align} \label{eqn:wedgemultiarmeq}
P^\delta(\mathcal{C}_{\Omega, z}^k(r,R)) \lesssim
\left(\frac{r}{R}\right)^{k(k+1)/12\theta - \eps},
\end{align}
where $2\pi \theta$ is the angle formed by $\partial \Omega$ at $z$.
\end{prop}

\begin{proof}
  Define $a_{k,\delta}^\theta(r,R)$ to be the probability of $k$ disjoint crossings
  of alternating color from inner to outer radius in $\{z\,:\, \arg z \in
  (0,2\pi \theta)\text{ and }r < |z|<R\}$. In \cite{SW}, it is shown that
\begin{equation} \label{eqn:sle_exponents}
\lim_{\delta \to 0} a_{k,\delta}^{1/2}(1,R) = R^{-k(k+1)/6+o(1)},
\end{equation}
using the convergence of the percolation exploration path to $\SLE_6$. By
the invariance of the law of $\SLE_6$ under the conformal map $z\mapsto
z^{2\theta}$, we conclude that \eqref{eqn:sle_exponents} generalizes
to
\[
\lim_{\delta \to 0} a_{k,\delta}^\theta(1,R) = R^{-k(k+1)/12\theta+o(1)}.
\]
The following multiplicative property is also used in \cite{SW}: for all
$k<r\leq r'\leq r''$, we have
\begin{equation} \label{eqn:submult}
a^{1/2}_{k,\delta}(r,r'') \leq a^{1/2}_{k,\delta}(r,r')a^{1/2}_{k,\delta}(r',r'').
\end{equation}
This inequality still holds with $1/2$ replaced by $\theta$. The proof in
\cite{SW} for the case $\theta = 1/2$ relies only on these two facts and
therefore generalizes to \eqref{eqn:wedgemultiarmeq} for the sector domain
$\{z\,:\, \arg z \in (0,2\pi\theta)\}$. The extension of this result to
piecewise real-analytic Jordan domains with positive interior angles is
obtained by following the same argument carried out in
Corollary~\ref{cor:smoothtwothree} for $\theta = 1/2$.
\end{proof}

\section{Proof of Main Theorem} \label{mainthm}

\subsection{Background and set-up}
We begin by recalling few definitions and facts from complex analysis and
differential geometry. See \cite{Ahlfors}, \cite{S86}, and \cite{L03} for more
details. If $a,b\in\bbC$ are linearly independent over $\bbR$ and $P$
is a parallelogram with vertex set $\{0,a,b,a+b\}$, then
a function $f:P\to \bbC\cup\{\infty\}$ is said to be doubly-periodic if
$f(z+a)=f(z)$ for $z$ on the segment from 0 to $b$ and
$f(z+b)=f(z)$ for all $z$ on the segment from 0 to $a$. If $f$ is
continuous, then such a function may be extended by periodicity to a
continuous function defined on $\bbC$. An elliptic function is a
doubly-periodic function whose extension to $\bbC$ is analytic outside of a
set of isolated poles. Given distinct points $p_1,p_2\in P$,
there exists an elliptic function $f$ with simple poles at $p_1,p_2$ (and
no other poles) \cite[Proposition~3.4]{S86}. One way to obtain such a function
is to define the Weierstrass product
\[
\sigma(\zeta) = \zeta \prod_{\substack{(j,k)\in \bbZ^2 \\ (j,k)\neq (0,0)}} \left(1-\frac{\zeta}{a
    j + b k}\right) \exp\left(\frac{\zeta}{a
    j + b k} + \frac{\zeta^2}{2(a
    j + b k)^2}\right)
\]
and set
\begin{equation} \label{eq:poles}
f(\zeta) = \frac{\sigma((\zeta-(p_1+p_2)/2))^2}{\sigma(\zeta-p_1)\sigma(\zeta-p_2).}
\end{equation}

We recall the definitions of the differential forms $d\zeta = dx + i\, dy$
and $d\overline{\zeta}=dx - i\, dy$. Note that $d\bar \zeta \wedge d\zeta =
2i dA$, where $dA$ is the two-dimensional area measure and $\wedge$ is the
usual wedge product. Recall that the exterior derivative $d$ maps $k$-forms
to $(k+1)$-forms and satisfies
\begin{equation} \label{eq:diff_properties}
df=\partial f \,d\zeta + \bar\partial fd\bar\zeta, \text{ and } d(f d\zeta) =
df\wedge d\zeta
\end{equation}
for all smooth functions $f$.

Let $\phi:(\Omega,x(1),x(\tau),x(\tau^2)) \to (T,1,\tau,\tau^2)$ be the
unique conformal map from $\Omega$ to the equilateral triangle $T$ with
vertices $1,\tau,$ and $\tau^2$ which maps $x(\tau^k)$ to $\tau^k$ for
$k\in\{0,1,2\}$. Let $\delta>0$ be small and define $\Omega_{\text{edge}}$
to be such that $\Omega\setminus \Omega_{\text{edge}}$ is the set of all
hexagonal faces of $(\Omega_\delta)^*$ completely contained in
$\Omega$. Let $T_{\text{edge}}$ be the image of $\Omega_{\text{edge}}$
under $\phi$.

We modify $G^\delta$ to obtain a function $\tilde{G}^\delta$ for which the
lattice points on the boundary of $\Omega\setminus \Omega_{\text{edge}}$
are mapped to the boundary of $T$. Specifically, we set
\[
\tilde G^\delta(z) = \begin{cases}
\tau^k & z\text{ is adjacent to } x(\tau^k) \\
\proj\left(G^\delta(z),[\tau^k,\tau^{k+1}]\right) & \text{if }z\text{ is not
  adjacent to } x(\tau^k) \\
 &
\text{\hspace{1cm}but is adjacent to }[x(\tau^{k+1}),x(\tau^{k+2})] \\
G^\delta(z) & \text{otherwise,}
\end{cases}
\]
where we are using the notation $\proj(z,L)$ for the projection of a
complex number $z$ onto the line $L\subset \bbC$.  Now linearly interpolate
to extend $\tilde{G}^\delta$ to a function on $\Omega$, and define $J:T\to
T$ by $J(w) = \tilde{G}^\delta(\phi^{-1}(w))$.

Schwarz-reflect 17 times to extend $J$ to the parallelogram $P$ in
Figure~\ref{fig:parallelogram}. For example, if $r$ is the reflection across
the line through $1$ and $\tau$, then for $w$ in the triangle $r(T)$, we
define $J(w)=r\circ J \circ r(w)$. Define an elliptic function $g_w$ via
\eqref{eq:poles} with period parallelogram $P$ and poles at
$p_1=w_0\colonequals (1+\tau+\tau^2)/3$ and $p_2=w$ varying over the grey
triangle $K$ in Figure \ref{fig:parallelogram}.

\begin{pictures}
\begin{figure}[h]
\begin{center}
\includegraphics{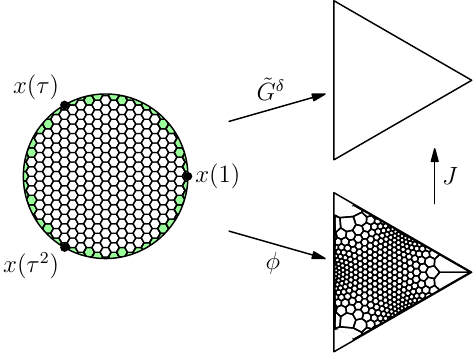}

\begin{minipage}{\linewidth}
  \caption{The function $J$ is defined as the composition of
    $\tilde{G}^\delta$ with the inverse of the Riemann map from $\Omega$ to
    the triangle. The region $T_{\text{edge}}$ is the image under the
    conformal map $\phi$ from $\Omega$ to $T$ of the region
    $\Omega_{\text{edge}}$, shown in green.} \label{Jfig}
\end{minipage}
\end{center}
\end{figure}
\end{pictures}

We will also need a result from the theory of Sobolev spaces. If $U\subset
\bbR^2$ is a bounded domain, and $1\leq p < \infty$, we define the Sobolev
space $W^{1,p}(U)$ to be the set of all functions $u:U\to \bbR$ such that the
weak partial derivatives of $u$, $\frac{\partial u}{\partial x}$, and
$\frac{\partial u}{\partial y}$ are in $L^p(U)$; see \cite{Evans} for more
details. We equip $W^{1,p}(U)$ with the norm
\[\|u\|_{W^{1,p}}\colonequals
\left\|u\right\|_{L^p(U)}+\left\|\frac{\partial u}{\partial
    x}\right\|_{L^p(U)}+\left\|\frac{\partial u}{\partial
    y}\right\|_{L^p(U)}.\] Denote by $\id$ the identity function from $P$
to $P$, and define $C^\infty(P)$ to be the set of smooth, real-valued
functions from $P$. Since $J$ is piecewise-affine on $P$, the real and
imaginary parts of $J$ are in $W^{1,1}(P)$.  Since $J$ is defined so that
$J:T\to T$ takes vertices to vertices and boundary segments to boundary
segments, $J-\id$ is continuous and doubly-periodic.  Since smooth
functions are dense in $W^{1,1}(P)$ and $L^\infty(P)$ \cite{Evans}, for
each $\eps>0$ we obtain a pair of smooth functions $Q_1,Q_2 \in
C^\infty(P)$ such that
 \begin{align} \nonumber
  | Q(w) - (J(w) - w)| &< \eps \text{ for all }w\in P,\\ \label{eq:smooth}
   \|Q_1 - \Re(J - \text{id})\|_{W^{1,1}}  &< \eps,\text{ and} \\ \nonumber
   \|Q_2 - \Im(J - \text{id})\|_{W^{1,1}}  &< \eps,
 \end{align}
 where $Q = Q_1 + iQ_2$ (for see \cite{Evans} \S 5.3.3 and \S C.5, for
 example). Defining $Q_1$ and $Q_2$ to be bump function convolutions, we
 arrange for $Q_1$ and $Q_2$ to inherit periodicity from $J-\id$.  We note
 that by choosing $\eps$ sufficiently small in \eqref{eq:smooth}, we can
 for every $\eps'>0$ choose $Q$ so that
\begin{equation} \label{eq:approxg}
\int_P \left|\partial Q - \partial(J-\id)\right|\left| g_w\right| \,dA < \eps',
\end{equation}
where $dA$ refers to two-dimensional Lebesgue measure. One way to see
this is to define $f(z)=\partial Q(z) - \partial(J(z) - z)$ and note that
for $R>0$, we have
\begin{align} \label{eq:split}
\|fg\|_{L^1} & \leq  R \|f\|_{L^1} + \|f\|_{L^\infty}\|\mathbf{1}_{\{|g| > R\}} g\|_{L^1}.
\end{align}
By the dominated convergence theorem, we may choose $R$ sufficiently large
that the second term on the right-hand side is less than $\eps'/2$. Once $R$
is chosen, we may choose $Q$ so that $\|f\|_{L_1}\leq \eps'/(2R)$, by
\eqref{eq:smooth}. Then \eqref{eq:approxg} follows from \eqref{eq:split}.

\begin{pictures}
\begin{figure}[h]
\begin{center}
\includegraphics{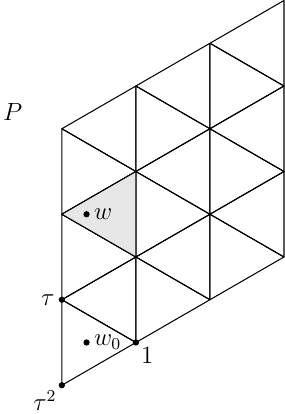}
\begin{minipage}{\linewidth}
  \caption{We extend $J(w)$ to a function on the parallelogram $P$, which
    is a union of 18 small triangles. The elliptic function $g_w:P \to
    \bbChat$ has poles at $w_0$ fixed and $w$ varying in the gray
    region.} \label{fig:parallelogram}
\end{minipage}
\end{center}
\end{figure}
\end{pictures}

\subsection{Proof of main theorems}
%We can now focus on the proof of Theorem~\ref{thm:general}

\begin{proof}[Proof of Theorem \ref{thm:general}]
  The following calculation is similar to the proof of the Cauchy integral
  formula, but with two key changes: we keep track of the $\bar\partial $
  term, and we use the elliptic function $g_w$ in place of the usual kernel
  $\zeta\mapsto 1/\zeta$. Choose $r>0$ sufficiently small that the balls
  $B_1$ and $B_2$ of radius $r$ around $w_0$ and $w$ are disjoint, and
  apply Stokes' theorem to the region $ P\setminus (B_1\cup B_2) $ to
  obtain that for smooth, complex-valued, periodic functions $Q$ on $P$, we
  have
\begin{align*}
  -\int_{|\zeta-w|= r} Q(\zeta) g_w(\zeta)\,d\zeta-&\int_{|\zeta-w_0|=
  r} Q(\zeta) g_w(\zeta)\,d\zeta = \int_{P\setminus(B_1 \cup B_2)} d(Q g_w d\zeta).
\end{align*}
Note that the integral around $\partial \Omega$ vanishes by
periodicity. Applying \eqref{eq:diff_properties} and the product rule, we
obtain
\begin{align*}
  \int_{P\setminus(B_1 \cup B_2)} d(Q g_w d\zeta) &= \int_{P\setminus(B_1
    \cup B_2)} \left[\struts{6}(\partial Q d\zeta + \bar\partial Q
    d\bar\zeta) g_w + (\partial g_w d\zeta + \bar\partial g_w
    d\bar\zeta) Q \right]  \wedge d\zeta \\
  &= \int_{P\setminus(B_1 \cup B_2)} \bar\partial Q(\zeta) g_w(\zeta)
  \,d\bar\zeta \wedge d\zeta.
\end{align*}
Let $Q$ be a smooth, complex-valued, periodic function on $P$ such that
\eqref{eq:smooth} and \eqref{eq:approxg} are satisfied with $\eps = \eps' =
\delta^{100}$, say.  Since $Q$ is bounded and $g$ has an integrable pole at
$\zeta$, we can take $r\to 0$ and apply the dominated convergence
theorem. We obtain
\begin{equation} \label{Stokes}
2\pi iQ(w) \Res (g_w,w) + 2\pi iQ(w_0) \Res (g_w,w_0) =
2i \int_P \bar \partial Q(\zeta) g_w(\zeta)\, dA(\zeta),
\end{equation}
where $dA(\zeta) = dx \, dy$ is notation for the area differential. The key
step of the proof is to bound the right-hand side of (\ref{Stokes}) by
$O(\delta^c)$. To do this, we first consider $J$ in place of $Q$, and we
estimate the integral over the regions $T\setminus T_{\text{edge}}$ and
$T_{\text{edge}}$ separately. We postpone the details of these calculations
to the following section, along with stronger lemma statements (Lemmas
\ref{lem:a_outside} and \ref{lem:a_multiscale}).
\begin{lemma} \label{lem:outside} There exists $c>0$ so that
\begin{equation} \label{gdA}
\int_{T_\text{edge}} \bar \partial J(\zeta) g_w(\zeta) dA(\zeta)
 = O(\delta^{c}),
\end{equation}
where the implied constants depend only on the three-pointed domain.
\end{lemma}
\begin{lemma} \label{lem:multiscale} There exists $c>0$ so that
\begin{equation} \label{gdA2}
\int_{T\setminus T_{\text{edge}}} \bar \partial J(\zeta)g_w(\zeta)
  dA(\zeta)  =O( \delta^c),
\end{equation}
where the implied constants depend only on the three-pointed domain.
\end{lemma}

Since $\Res (g_w,w)$ is a continuous function of $w$ with no
zeros in $K$, there exists $C>0$ such that
\[0<C^{-1} < \Res (g_w,w) < C <\infty, \qquad \forall w\in K, \] and
similarly for the residue at $w_0$. Therefore, \eqref{gdA2} implies that
$Q(w)$ is within $O(\delta^c)$ of a constant function, as $w$ ranges over
the gray triangle shown in Figure~\ref{fig:parallelogram}.  By considering $w$
to be one of the vertices of the gray triangle (so that $J(w)-w=0$), we see
that this constant function is $O(\delta^c)$. We conclude that $Q(w) =
O(\delta^c)$.  By (\ref{eq:smooth}), this implies $J(w)-w=O(\delta^c)$. By
definition, this is equivalent to $\tilde{G}^\delta(z)-\phi(z) =
O(\delta^c)$. The theorem follows, since $\tilde{G}^\delta$ agrees with
$G^\delta$ except on the outermost layer of lattice points.
\end{proof}

We combine the rate of convergence for $H_1+\tau H_\tau+\tau^2H_{\tau^2}$
with the rate of convergence for $H_1+ H_\tau+H_{\tau^2}$ near $\partial
\Omega$ to prove the rate of convergence of the crossing probabilities.

\begin{proof}[Proof of Theorem~\ref{thm:cardyrate}]
  Let $z\in [x(1),x(\tau)]$. First we note that
  $H^\delta_{\tau^2}(z)=O(\delta^c)$ by Proposition \ref{prop:boundary}. Hence, by
  Theorem~\ref{thm:general},
     \begin{align*}
%    H_1 + \tau H_\tau + \tau^2 H_{\tau^2} (p) &= |p|\tau + (1-|p|) + O(\delta^c)  \quad \implies \\
%    H_1(z) + \tau H_\tau(z) &= |z|\tau + (1-|z|) + O(\delta^c).
    H^\delta_1(z) + \tau H^\delta_\tau(z) &= \phi(z) + O(\delta^c) \, .
  \end{align*}
We also have that
$$     H^\delta_1(z) + H^\delta_\tau(z) = 1 + O(\delta^c) \, ,$$
since $S^\delta(z) = 1 + O(\delta^c)$ by Proposition~\ref{prop:boundary}
(ii).  Since the vectors $(1,1),(1,\tau)\in \bbC^2$ are linearly
independent, this concludes the proof.
%
%    \[
%  (H_1,H_\tau)=(|p|,1-|p|)+ O(\delta^{c}). \qedhere
%  \]
\end{proof}

\section{Bounding the error integral} \label{appendix}

\subsection{Piecewise analytic Jordan domains}

In this section, we prove the two lemmas used in the proof of the main
theorem. We often treat the conformal map $\phi(z)$ like a power of $z$
when $z$ is near a corner of the domain $\Omega$. To make this precise, we
use the following theorem from the conformal map literature \cite{L57}.
\begin{theorem} \label{thm:confmap}
  If $\Omega$ is a Jordan domain part of whose boundary consists of two
  analytic arcs meeting at a positive angle $2\pi \alpha$ at the origin, and
  if $\phi:\Omega \to \bbH$ is a Riemann map sending 0 to 0, then
  there exists a neighborhood $B$ of the origin and continuous
  functions $\rho_1,\rho_2:B\cap \overline{\Omega} \to \bbC$ and
  $\rho_3,\rho_4:\phi(B\cap \overline{\Omega}) \to \bbC$ for which
\begin{align*}
\phi(z) &= z^{1/(2\alpha)} \rho_1(z), \qquad \phi'(z) = z^{1/(2\alpha) - 1}
\rho_2(z), \\ \phi^{-1}(z) &= z^{2\alpha} \rho_3(z),\text{ and} \qquad
(\phi^{-1})'(z) = z^{2\alpha-1} \rho_4(z)
\end{align*}
and $\rho_i(0)\neq 0$ for $i\in \{1,2,3,4\}$.
\end{theorem}
We choose a collection $\mathcal{B}$ of disks covering the boundary of
$\Omega$ as follows (see Figure~\ref{fig:boundarydisks}). For each
$z\in \partial \Omega$, choose a disk $B(z)$ centered at $z$ and small
enough that the boundary arc (or arcs) containing $z$ admits a Taylor
expansion in $B(z)$. If necessary, shrink $B(z)$ so that $\partial \Omega$
is well-approximated by its tangent (or tangents, if $z$ is a corner point)
in $B(z)$, in the sense of Propositions~\ref{prop:cont_exp} and
\ref{cor:smoothwedgemultiarm}. If necessary, shrink $B(z)$ once more to
ensure that $\Omega \cap B(z)$ has one component.  From this collection of
open disks, extract a finite subcover $\calB = (B_j)_{j=1}^p$ of $\partial
\Omega$ containing $\calB_{\text{corners}} = \{B(z)\,:\,z\text{ is a corner
  point}\}$. Then $\calB$ is an annular region whose interior has positive
distance from $\partial \Omega$. Thus, for all sufficiently small $\delta$,
$\calB$ covers $\Omega_{\text{edge}}$. Note that this cover has been chosen
in a manner which depends only on $\Omega$ and $\eps$, and in particular is
independent of $\delta$.

Throughout our discussion, we permit the constants in statements involving
asymptotic notation to depend only on the three-pointed domain. We also use
$C$ to represent an arbitrary constant which depends only on the
three-pointed domain. When working with the variable $\eps$, we will
frequently relabel small constant multiples of $\eps$ as $\eps$
from one line to the next.

\begin{pictures}
\begin{figure}
  \centering
  \includegraphics[width=6cm]{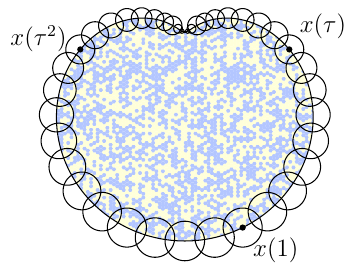}
  \caption{We cover the boundary with finitely many small disks, so that
    the boundary is approximately straight in each disk. Moreover, we
    ensure that every corner and every marked point is centered at one of
    these disks. More disks are required in regions of high curvature, as
    illustrated here for a domain bounded by a limaçon.} \label{fig:boundarydisks}
\end{figure}
\end{pictures}

\begin{lemma} \label{lem:a_outside} Let $J, g_w$ be as in
  Section~\ref{mainthm}, and suppose that the angle measures at marked
  points are $2\pi \alpha_i$ for $i=1,2,3$, and remaining angles are $2\pi
  \beta_j$ for $j=1,2,\ldots,n$. For every $\eps>0$,
\begin{equation} \label{gdAA}
\int_{T_\text{edge}} \bar \partial J(\zeta) g_w(\zeta) dA(\zeta) \lesssim \delta^{\min_{i,j}\left(1,\frac{1}{6\alpha_i},\frac{1}{2\beta_j}\right)-\eps},
\end{equation}
where the implied constants depend only on $\eps$ and the three-pointed domain.
\end{lemma}

\begin{proof}[Proof of Lemma.]
   Let $\calB$ be as described above. Since the number of disks in $\calB$ is bounded independently of $\delta$, it suffices to demonstrate that
  \eqref{gdAA} holds for each one.  Let $B\in \calB$, and let $\pi
  \beta$ be the angle formed by $\partial \Omega$ center of $B$.
 \begin{pictures}
    \begin{figure}
    \centering
    \includegraphics[width=6cm]{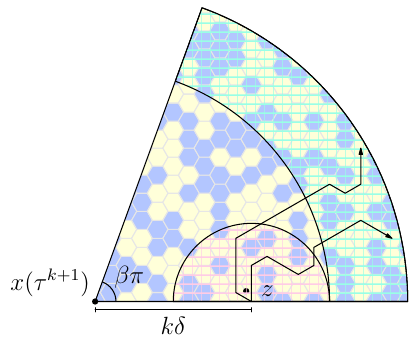}
    \caption{ If $z$ is adjacent to the side
      $[x(\tau^{k+1}),x(\tau^{k+2})]$, then the distance from $G^\delta(z)$
      to $\partial T$ is equal to the probability $H_{\tau^k}(z)$. This
      probability is bounded by that of a two-arm half-plane event with
      radius $k\delta$ and the two-arm $\beta$-annulus event with inner radius
      $2k\delta$ and constant-order outer radius.}
    \label{fig:integrating}
  \end{figure}
  \end{pictures}

  To bound $\left|\int_{T_\text{edge}\cap \phi(\Omega \cap B)}
    \bar \partial J(\zeta) g_w(\zeta) dA(\zeta) \right|$, we index all the
  faces $\{F_k\}$ intersecting $\partial \Omega$ in such a way that the
  distance from $F_k$ to the center of $B$ is $\asymp k\delta$ for all $k$;
  this is possible since $\partial \Omega$ is piecewise smooth. We will
  bound the integral over each $F_k$ and then sum over $k$ (see
  Figure~\ref{fig:integrating}).  Let $\zeta\in T_{\text{edge}}\cap
  \phi(F_k)$ and suppose that $[\tau,\tau^2]$ is the closest boundary arc.
%  Suppose that the boundary arc
%  $[\tau,\tau^2]$ is the closest to $\zeta\in T_{\text{edge}}$, and suppose
%  $\zeta \in \phi(F_k)$.
   We rewrite
\begin{equation} \label{eq:partialJ}
\bar \partial J(\zeta) = \bar \partial \tilde G^\delta (\phi^{-1}(\zeta)) (\phi^{-1})'(\zeta),
\end{equation}
and we define $z=\phi^{-1}(\zeta)$.  First we bound $\bar \partial \tilde
G^\delta(z)$. In modifying $G^\delta(z)$ to obtain
$\tilde{G}^\delta(z)$, the image of $z$ has to be moved no farther
than $H_1(z) = \bbP(E_1(z))$, by the definition of
$\tilde{G}^\delta(z)$. The event $E_1(z)$ entails a two-arm
half-disk crossing and a two-arm $\beta$-annulus crossing (see
Figure~\ref{fig:integrating}). Since these events occur in disjoint
regions, they are independent and we can bound $\bbP(E_1(z))$ by the
product of their probabilities. By Corollary~\ref{cor:smoothtwothree}, the
two-arm half-plane exponent in $\Omega$, is 1 and by
Proposition~\ref{cor:smoothwedgemultiarm} the two-arm $\beta$-annulus
exponent is $1/2\beta$. Thus the probability of $E_1(z)$ is at most
$(k\delta)^{1/2\beta -\eps} (1/k)^{1-\eps}$. Hence for $z+\eta$ in the
outermost layer and $z$ a neighbor of $z+\eta$, we have
\begin{align} \nonumber
 \frac{1}{\delta}&\left( \tilde{G}^\delta(z+\eta)
    -\tilde{G}^\delta (z) \right) \\ \label{eq:edge}
  &= \frac{1}{\delta}\left( \tilde{G}^\delta(z+\eta) -
    G^\delta(z+\eta) + G^\delta(z+\eta)  - G^\delta(z) + G^\delta(z) -
    \tilde{G}^\delta (z) \right) \\ \nonumber
&\leq \frac{1}{\delta}\left( (k\delta)^{1/2\beta -\eps}
  (1/k)^{1-\eps} +
  G^\delta(z+\eta) - G^\delta(z) \right). \\ \nonumber
&\asymp \delta^{-1} (k\delta)^{1/2\beta -\eps}
  (1/k)^{1-\eps}.
\end{align}
In the last step we use a shifted domain trick (see the proof of the second
inequality in Proposition~\ref{prop:CR} and Figure~\ref{fig:fivearm}) and
apply the trivial inequality $\bbP(A\setminus B) \leq \bbP(A)$. Using
\eqref{eq:edge} to bound each term of the expression $\bar\partial
\tilde{G}^\delta(z)=\left( \frac{\partial}{\partial \eta}\right.  -\left.
  \frac{1}{\tau}\frac{\partial}{\partial(\tau\eta)}
\right)\tilde{G}^\delta(z)$, we get $\bar\partial \tilde{G}^\delta(z)
\asymp \delta^{-1} (k\delta)^{1/2\beta -\eps} (1/k)^{1-\eps}$.

  We assume that the location $z$ of the pole is in the face nearest to the center of
$B$ (since that is the worst case) and also that the image of the center of
$B$ is not a vertex of the equilateral triangle $T$.
% We bound $\left|\int_{T_\text{edge}\cap \phi(\Omega\cap B_j)} \bar \partial J(\zeta)
%  g_w(\zeta) dA(\zeta)\right|$ by replacing the integrand with its supremum
%on each face $F_k$ and summing over $k$:
We obtain
\begin{align} \label{eq:supremum_estimation}
  \bigg|\int_{T_\text{edge}\cap \phi(\Omega\cap B_j)} & \bar \partial
  J(\zeta) g_w(\zeta) dA(\zeta)\bigg| \\ \nonumber &\leq \sum_{k=1}^{C/\delta}
  \sup_{z\in F_k} |\bar\partial \tilde{G}^\delta(z)
  (\phi^{-1})'(\phi(z))g_w(\phi^{-1}(z))| \text{area}(\phi(F_k))
\end{align}
by replacing the integrand with its supremum on each $F_k$ and
  summing over $k$. We use the estimate $\text{area}(\phi(F_k))
\lesssim \sup_{z\in F_k}|\phi'(z)|^2\delta^2$ and use
Theorem~\ref{thm:confmap} to estimate the factors involving $\phi$. We
bound the right-hand side of \eqref{eq:supremum_estimation} by
\begin{align} \label{curlybraces}
    &\hspace{0.2cm} \nonumber \lesssim \sum_{k=1}^{C/\delta}
  \overbrace{\delta^{-1}(k\delta)^{1/2\beta -\eps}
  (1/k)^{1-\eps}}^{\bar\partial
  \tilde{G}^\delta}\overbrace{(k\delta)^{1-1/2\beta}}^{(\phi^{-1})'(\phi(z))}\overbrace{(k\delta)^{-1/2\beta}}^{g_w}\overbrace{\delta^2
  (k\delta)^{2/2\beta-2}}^{\text{area}(\phi(F_{k}))}\\ \nonumber
&\hspace{0.2cm}\asymp \delta^{1-\eps} \left(\sum_{k=1}^{C/\delta}
(k\delta)^{1/2\beta-2} \delta \right)\\ \nonumber
&\hspace{0.2cm}\asymp \left\{
\begin{array}{cl}
  \delta^{1-\eps} & \text{if } 2\beta \leq 1 \\ \nonumber
  \delta^{1/2\beta -\eps} & \text{if } 2\beta > 1.
\end{array}
\right.
\end{align}
% LEAVE REMARK IN PAPER BUT COMMENTED OUT
We have evaluated the sum by noting that the factor in parentheses is a
convergent Riemann sum when the exponent is at least $-1$. When the
exponent is less than $-1$, the summation over $k$ gives a constant factor,
leaving the contributions of the powers of $\delta$.

If the center of $B_j$ is a marked point, the proof is essentially the same
and the net effect is to replace $1/2\beta$ with $1/6\alpha$ throughout the
calculation. These replacements are justified either by fewer percolation
arms (when the exponent appears in an arm event estimate), or by the angle
of $\pi/3$ at the vertices of the triangle $T$ (when the exponent appears
because of the conformal map $\phi$).
\end{proof}

\begin{remark} \label{rem:sle_free} We can remove the dependence on
    SLE by using Smirnov's theorem to estimate one-arm $\beta$-annulus
    probabilities (instead of using
    Proposition~\ref{cor:smoothwedgemultiarm}). The result is that we
    obtain \eqref{gdAA} with the right-hand side replaced by
  \[\delta^{\min_{i,j}\left(1,\frac{1}{6\alpha_i},\frac{1}{6\beta_j}\right)-\eps}.\]
 % By substituting the weaker one-arm
 %  $\beta$-annulus estimates (which removes the dependence on SLE results),
 %  we obtain \eqref{gdAA} with the right-hand side replaced by
 %  \[\delta^{\min_{i,j}\left(1,\frac{1}{6\alpha_i},\frac{1}{6\beta_j}\right)-\eps}.\]
\end{remark}

\begin{lemma} \label{lem:a_multiscale} Let $J, g_w, \{\alpha_i\},
  \{\beta_j\}$ be as in the statement of Lemma~\ref{lem:a_outside}. Let
  $c_3=2/3$ be the 3-arm whole-plane exponent. Then
  \begin{align} \label{eq:int} \int_{T\setminus T_\text{edge}}
    \bar \partial J(\zeta) g_w(\zeta) dA(\zeta)
    &\lesssim \delta^{\min_{i,j}\left(c_3,\frac{1}{6\alpha_i},\frac{1}{2\beta_j}\right)-\eps},
  \end{align}
  where the implied constants depend only on $\eps$ and the three-pointed
  domain.
\end{lemma}

\begin{proof}[Proof of Lemma.]
  We will use Proposition~\ref{prop:CR} to bound $\bar\partial G$. Let
  $\calB$ be as above and note that $\text{dist}(\partial \Omega, \Omega
  \setminus \bigcup \calB) > 0$ by the discussion preceding
  Lemma~\ref{lem:a_outside}.

  We first handle $\Omega \setminus \bigcup \mathcal{B}$. Suppose that one
  of the five-arm events of Figure~\ref{fig:fivearm} occurs, say
  $E_{1,1,1}^{\text{five arm}}(z)$. Let $b$ be the point nearest
  $x(\tau^2)$ where a blue arm touches down in the shifted domain, and
  let $s$ be the number of lattice units along the boundary from $b$ to
  $x(\tau^2)$.  When $z\notin \bigcup \mathcal{B}$ (see
  Figure~\ref{fig:middle}), $z$ is well away from the boundary thus we
  note that such a five arm event entails the existence of:
\begin{enumerate}
\item a $3$-arm whole-plane event in alternating colors at $z$, in a ball
  of radius $\Theta(1)$,
\item a $3$-arm half-annulus event of alternating colors originating at $b$,
  in a semi-circle of radius $s\delta/2$, and
\item a $2$-arm half-annulus event in an annulus of inner radius
  $s\delta/2$ and outer radius~$\Theta(1)$.
\end{enumerate}
Since the derivative of the conformal map is bounded above and below for
$z$ away from the boundary, we can ignore the contribution of
$\phi'(\phi^{-1}(z))$ in \eqref{eq:partialJ} and calculate
  \begin{align*}
    |\bar\partial J(z)| &\lesssim \delta^{-1}\sum_{s=1}^{C/\delta}
    \overbrace{(\delta^{c_3-\eps})}^{\text{3-arm
        whole-plane}}\times\overbrace{(1/s)^2}^{\text{3-arm
        half-plane}}\times\overbrace{(s\delta)}^{\text{2-arm half-ann.}}
    \\ &\lesssim \delta^{c_3-\eps}.
\end{align*}
Hence we have
\begin{align*}
  \left|\int_{\phi(\Omega \setminus \calB)} \bar \partial J(\zeta)
    g_w(\zeta) dA(\zeta) \right| \lesssim \delta^{c_3-\eps} \int_{\phi(\Omega
    \setminus \calB)} |g_w(\zeta)|\,dA(\zeta) \lesssim \delta^{c_3-\eps},
\end{align*}
since a simple pole is integrable with respect to area measure.

\begin{pictures}
\begin{figure}
  \centering
  \includegraphics[width=6cm]{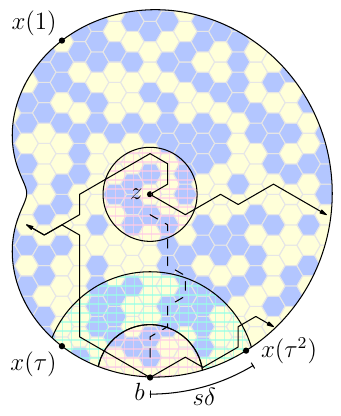}
  \caption{To bound the probability of the five-arm difference event
    described in Proposition~\ref{prop:CR}, we consider three regions which
    contain two-arm or three-arm crossing events (these regions are shown
    in green and red, respectively).} \label{fig:middle}
\end{figure}
\end{pictures}

To bound the integral of the union of the balls in $\calB$, we handle each
$B\in \calB$ separately. We first consider a ball centered at a marked
corner, say $x(\tau)$.  Once again, for each $z$ and each percolation
configuration, we define $b\in \partial \Omega$ to be the point nearest
$x(\tau^2)$ at which a blue arm from $z$ touches down in the shifted
domain. This time we let $s$ be the graph distance from $b$ to the boundary
point $z_{\text{foot}}$ nearest to $z$ (see Figure~\ref{fig:cornerlabels})
and index the faces $F_{n,k}$ in such a way that if $z\in F_{n,k}$,
$|x(\tau)-z|\asymp k\delta$ and $\dist(\partial
\Omega,z)\asymp n\delta$. As above, we bound $|\bar \partial
\tilde{G}^\delta(z)|$ using percolation arm estimates in each hexagonal
face and sum over all the faces in $\phi(\Omega \cap B)$.  By symmetry, it
suffices to sum over only the faces which are closer to the boundary arc
$[x(\tau),x(\tau^2)]$ than to the boundary arc
$[x(1),x(\tau)]$.

\begin{pictures}
\begin{figure}
\includegraphics[width=0.4\textwidth]{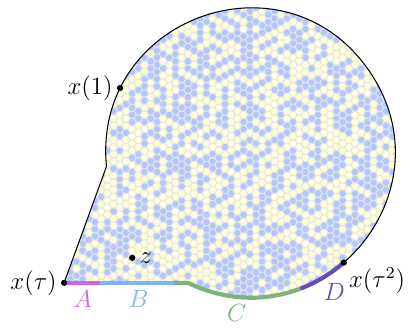}
\caption{We sum over the possible locations for $b$, considering the cases
  $b\in A$, $b\in B$, $b\in C$, and $b\in D$ separately.} \label{fig:bregions}
\end{figure}
\end{pictures}

Suppose that the corner at $B$ is one of the three marked points and has
interior angle $\alpha \pi$. We bound $|\bar \partial \tilde{G}^\delta|$ by
summing over all possible locations for $b$. We consider four cases:
\begin{itemize}
\item Case {A}: $b$ is closest to the corner at $x(\tau)$ (Figure~\ref{subfig:marked_bcorner}),
\item Case {B}: $b$ is within $k/2$ units of
$z_\text{foot}$ (Figure~\ref{subfig:marked_bnotcorner}),
\item Case {C}: $b$ is more than $k/2$ units to the right of $z_\text{foot}$ but closer to
$z_\text{foot}$ than to $x(\tau^2)$ (Figure~\ref{subfig:marked_b_otherside}), and
\item Case {D}: $b$ is closest to $x(\tau^2)$
  (Figure~\ref{subfig:marked_b_far}).
\end{itemize}
For simplicity, we assume that $[x(\tau),x(\tau^2)]$ is a real analytic arc (that
is, that there are no corners between $x(\tau)$ and $x(\tau^2)$). It will be
apparent that similar estimates hold when additional corners are accounted
for.

Denote by $P(z,b)$ the contribution to $\bar\partial G^\delta$ of the
five-arm event with missed connection at $b$ (see
Figure~\ref{fig:fivearm}). As in \eqref{eq:supremum_estimation}, we bound
the sum for Case A by a constant times
\begin{align*}
 \sum_{k=1}^{C/\delta} \sum_{n=1}^{Ck}\sum_{r=1}^{k/2}
  &\overbrace{\underbrace{\delta^{-1}}_{\bar\partial
    }\underbrace{n^{-c_3-\eps}}_{\text{3-arm
        disk.}}\underbrace{r^{-2}}_{\text{3-arm
        half-disk.}}\underbrace{\left(\frac{r}{k}\right)^{1/2\alpha-\eps}}_{\text{2-arm }\alpha\text{-ann.}}\underbrace{(k\delta)^{1/6\alpha-\eps}}_{\text{1-arm
      }\alpha\text{-ann.}}}^{P(z,b)} \\
&\hspace{3cm} \times
\overbrace{(k\delta)^{1-1/6\alpha}}^{(\phi^{-1})'(\phi(F_{n,k}))}
\overbrace{\delta^2
  (k\delta)^{1/3\alpha-2}}^{\text{area}(\phi(F_{n,k})))}
\overbrace{(k\delta)^{-1/6\alpha}}^{g_w} \\
% &\lesssim
% \left\{
% \begin{array}{cl}
%   \delta^{\min(c_3,1/6\alpha)-\eps} & \text{if }2\alpha \leq 1 \\
%   \delta^{1/6\alpha-\eps} & \text{if }2\alpha > 1.
% \end{array}
% \right. \\
&\lesssim \delta^{\min(c_3,1/6\alpha)-\eps}.
\end{align*}
We upper bound the contribution of Case B by a constant times
\begin{align*}
   \sum_{k=1}^{C/\delta} \sum_{n=1}^{Ck} \sum_{s=1}^{k/2}
  &\overbrace{\underbrace{\delta^{-1}}_{\bar\partial}\underbrace{n^{-c_3-\eps}}_{\text{3-arm
        disk}}\underbrace{\frac{1}{s^2+n^2}}_{\text{3-arm half
        disk}}\underbrace{\frac{\sqrt{s^2+n^2}}{k}}_{\text{2-arm half-ann.}}
    \underbrace{(k\delta)^{1/6\alpha-\eps}}_{\text{1-arm
      }\alpha\text{-ann.}}}^{P(z,b)} \\
&\hspace{3cm}\times
  \overbrace{(k\delta)^{1-1/6\alpha}}^{(\phi^{-1})'(\phi(F_{n,k}))}
  \overbrace{\delta^2
    (k\delta)^{1/3\alpha-2}}^{\text{area}(\phi(F_{n,k}))}
  \overbrace{(k\delta)^{-1/6\alpha}}^{g_w} \\
  % &\lesssim \left\{
% \begin{array}{cl}
%   \delta^{c_3-\eps} & \text{if } \alpha \leq 1/3c_3 \\
%   \delta^{1/6\alpha-\eps} & \text{if } \alpha > 1/3c_3.
% \end{array}
% \right.
 &\lesssim \delta^{\min(c_3,1/6\alpha)-\eps}.
\end{align*}
For Case C, we get
\begin{align*}
  \sum_{k=1}^{C/\delta} \sum_{n=1}^{Ck}\sum_{r=k/2}^{C/\delta}
  &\overbrace{\underbrace{\delta^{-1}}_{\bar\partial
    }\underbrace{n^{-c_3-\eps}}_{\text{3-arm
        disk.}}\underbrace{k^{-2}}_{\text{3-arm half-disk.}}}^{P(z,b)}
  (k\delta)^{1-1/6\alpha}\delta^2(k\delta)^{1/3\alpha-2}(k\delta)^{-1/6\alpha}
  \\
  &\lesssim \delta^{1/6\alpha-\eps}.
\end{align*}
For Case D, we denote by $2\pi \gamma$ the angle at $x(\tau^2)$ and by $t$
the number of lattice units from $x(\tau^2)$ to $b$. We obtain
\begin{align*}
 \sum_{k=1}^{C/\delta} \sum_{n=1}^{Ck}\sum_{t=1}^{C/\delta}
  &\overbrace{\underbrace{\delta^{-1}}_{\bar\partial
    }\underbrace{n^{-c_3-\eps}}_{\text{3-arm
        disk.}}\underbrace{t^{-2}}_{\text{3-arm
        half-disk.}}\underbrace{(t\delta)^{1/\gamma}}_{\text{2-arm }\gamma\text{
        ann}}}^{P(z,b)}
  (k\delta)^{1-1/6\gamma}\delta^2(k\delta)^{1/3\gamma-2}(k\delta)^{-1/6\gamma} \\
&\lesssim \delta^{1/2\gamma-\eps}.
\end{align*}

\begin{pictures}
\begin{figure}
  \subfigure[\label{subfig:marked_bcorner}]{\includegraphics[width=0.47\textwidth]{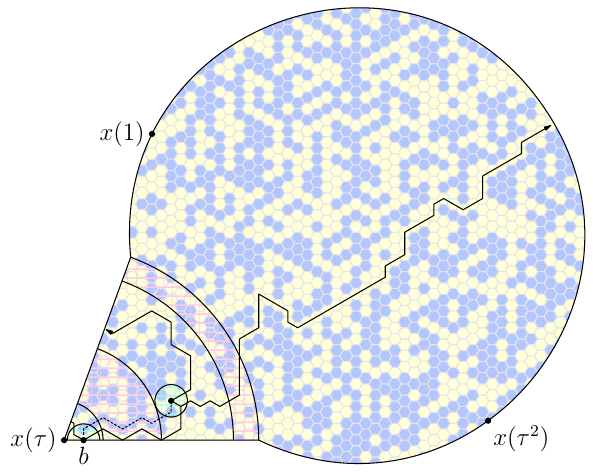}}
   \subfigure[\label{subfig:marked_bnotcorner}]{\includegraphics[width=0.47\textwidth]{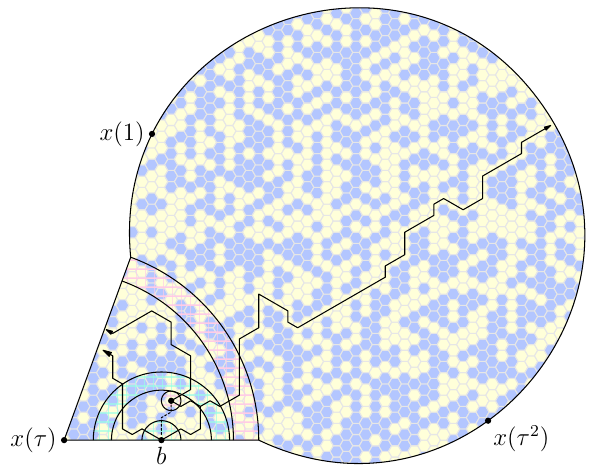}}  \\
  \subfigure[\label{subfig:marked_b_otherside}]{\includegraphics[width=0.47\textwidth]{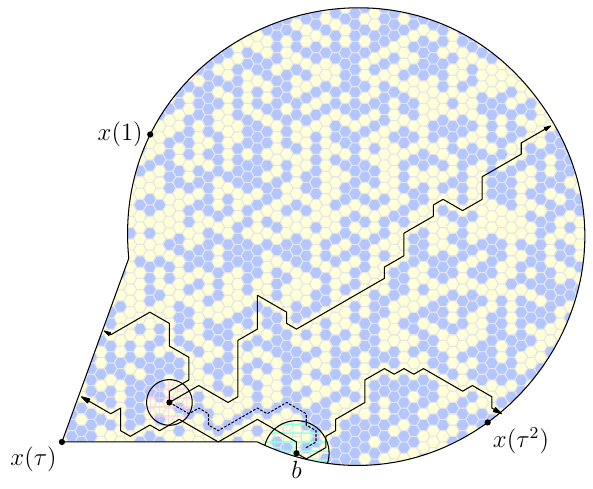}}
  \subfigure[\label{subfig:marked_b_far}]{\includegraphics[width=0.47\textwidth]{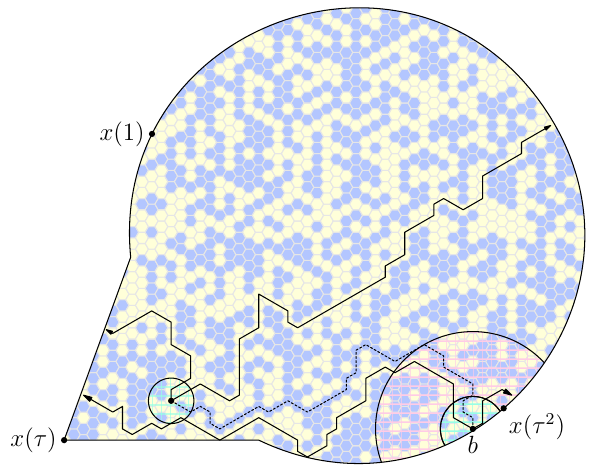}}
  \caption{Assuming that $z$ is near a marked corner, we have four cases to
    consider: (a) $b$ is close to $x(\tau)$, (b) $b$ is close to $z$, (c) $b$ is
    between $z$ and $x(\tau^2)$ but far from both, and (d) $b$ is close to
    $x(\tau^2)$. For a closer view of the corner with additional labels,
    see Figure~\ref{fig:cornerlabels}.  } \label{fig:markedcorner}
\end{figure}
\end{pictures}

\begin{pictures}
\begin{figure}
\includegraphics[width=0.45\textwidth]{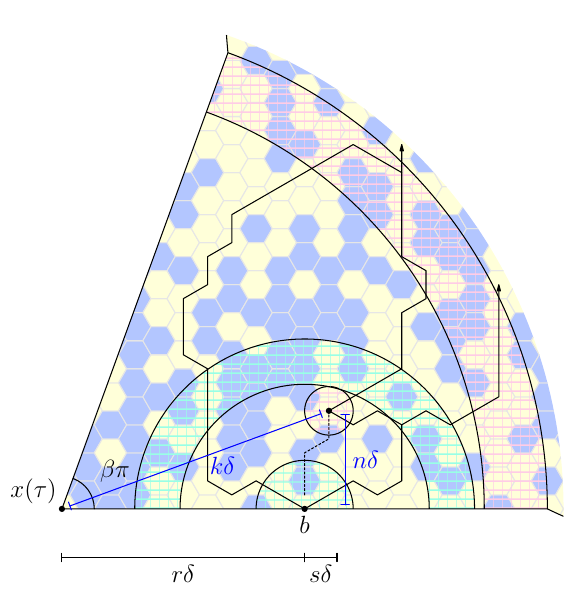}
\caption{A close-up view of the corner of
  Figure~\ref{subfig:marked_bnotcorner}, with labels illustrating the roles
  of $k$, $n$, $r$, and $s$. The faces are indexeed by $k$ and $n$ in such
  a way that the distance from $z$ to the corner is $\asymp k\delta$ and
  the distance from $z$ to $\partial \Omega$ is $\asymp
  n\delta$. Similarly, the faces intersecting the boundary are indexed so
  that the distance along the boundary from the corner to $b$ is $\asymp
  r\delta$ and the distance from $b$ to $z_{\text{foot}}$ is $\asymp
  s\delta$. } \label{fig:cornerlabels}
\end{figure}
\end{pictures}

The proofs for the bounds in a disk whose center is not marked are
essentially the same as these. As in the proof of
Lemma~\ref{lem:a_outside}, the net effect is to replace $1/6\alpha$ with
$1/2\beta$.
% These replacements are justified as in
%, namely either by additional percolation arms (when the exponent appears in an arm event estimate), or by the angle of $\pi/3$ at the vertices of the triangle $T$ (when the exponent appears because of the
%conformal map $\phi$).
%The net effect is to replace $1/6\alpha$ with $1/2\beta$ in the conclusion.
\end{proof}

\begin{remark}
  As in Remark~\ref{rem:sle_free}, we can remove the dependence on
    SLE by using Smirnov's theorem instead of
    Proposition~\ref{cor:smoothwedgemultiarm}, under the additional
  assumption that $\partial\Omega$ has no reflex angles (that is,
  $\max_{i,j}(\alpha_i,\beta_j)\leq 1/2$). By using the weaker one-arm
  $\beta$-annulus bound in place of the two-arm and three-arm bounds, we
  obtain \eqref{eq:int} with the right-hand side replaced by
  \[\delta^{\min_{i,j}\left(c_3,\frac{1}{6\alpha_i},\frac{1}{6\beta_j}\right)-\eps}.\]
  Without the help of SLE, our techniques break down in the presence of
  reflex angles.
\end{remark}

\subsection{Uniform bounds for half-annulus domains} \label{subsec:uniform_constant}

While the constants in Theorem~\ref{thm:cardyrate} generally depend on the
three-pointed domain, there are some classes of domains for which
Theorem~\ref{thm:cardyrate} holds with uniform constants. In preparation
for the proof of Theorem~\ref{thm:halfarm}, we obtain uniform constants
for a class of half-annulus domains with arbitrarily small ratio of inner to outer radius.

Let $\Omega_{r,R}\subset \bbH $ be the origin-centered half-annulus of
inner and outer radius $r$ and $R$, respectively. Let $T_{\text{unit}}$ be
the triangle with vertices $0,1,$ and $e^{i\pi/3}$, and define
$\phi_{r,R}:\Omega_{r,R} \to T_{\text{unit}}$ to be the conformal map
sending $-R$, $-r$, and $R$ to $e^{i\pi/3}, 0$, and $1$, respectively. For
$r\geq 0$, define $S_r = \{r e^{i\theta} \,: 0 \leq \theta \leq \pi\}$.

\begin{prop} \label{prop:uniform_constant}
  For all $0<c<c_3 = 2/3$ and $0<\delta\leq r\leq 1/2$, we have
  \begin{equation} \label{eq:uniform_constant}
    P^{\delta}(S_r \leftrightarrow S_1) - \phi_{r,1}(r)  =
    O(r^{-1/3}\delta^{c}) = O(\delta^{c-1/3}),
  \end{equation}
  where the implied constants depend only on $c$ and, in particular, are
  uniform over $r\in (0,1/2]$.
\end{prop}

\begin{remark}
  To ensure that the interval $(0,c_3-1/3)$ of possible exponents $c$ is
  nonempty, we need the $\SLE$ result that the three-arm whole-plane
  exponent $c_3$ is greater than $1/3$.
\end{remark}

\begin{proof}
  We proceed by modifying Lemmas~\ref{lem:a_outside} and
  \ref{lem:a_multiscale} to prove \eqref{gdAA} and \eqref{eq:int} with
  constants uniform over the domains $\Omega_{r,1}$. For $z\in \bbC$ and
  $\rho\geq 0$, let $B(z,\rho)$ be the disk of radius $\rho$ centered at
  $z$.  For the integral over $\Omega_{r,1} \setminus B(0,1/2)$ we obtain a
  bound of $O(\delta^{2/3-\eps})$ by Lemmas~\ref{lem:a_outside} and
  \ref{lem:a_multiscale}, so it suffices to consider the integral over
  $\Omega_{r,1} \cap B(0,1/2)$.

  Fix $\eps>0$, and determine $\alpha(\eps)$ from
  Proposition~\ref{prop:cont_exp}. Choose $\eta(\eps)$ small enough that
  $B(i,\eta)\setminus B(0,1) $ is contained in a sector of angle $\pi+\alpha$
  centered at $i$.  Cover $S_r$ with finitely many balls of radius $2r\eta$
  in such a way that $\bigcup_{w\in S_r} B(w,r\eta)$ is contained in the
  union $U$ of the balls. By Lemmas~\ref{lem:a_outside} and
  \ref{lem:a_multiscale} and rescaling \eqref{gdAA} and \eqref{eq:int} by a
  factor of $r$, we find that $\int_U|\bar\partial J g_w|
  \,dA=O(r^{-1/3}\delta^{c_3-\eps})$. So it remains to consider the
  integral over the annulus $A'\colonequals \{z\,:\, r(1+\eta) < |z| < 1/2\}$. We
  reduce further to considering the integral over the left half $\{z \in A'
  : \pi/2 < \text{arg}(z) < \pi\}$ of $A'$, since the contribution from the
  right half of $A'$ is smaller.  We compute this integral similarly to
  those in Lemmas~\ref{lem:a_outside} and \ref{lem:a_multiscale} (see
  Figure \ref{fig:uniform_constant}): we index the faces $F_{n,k}$ in such
  a way that $|F_{n,k}|-r \asymp k\delta$ and $\dist(F_{n,k},\bbR) \asymp
  n\delta$ and, for $z \in F_{n,k}$ we bound
  \begin{align*} &P(z,b) \lesssim \\
    &\underbrace{\delta^{-1}}_{\bar\partial }\underbrace{(n\wedge
      k)^{-c_3+\eps}}_{\text{3-arm
        disk.}}\underbrace{\left(\frac{\delta}{s\delta\wedge
          \eta r}\right)^{2-\eps}}_{\text{3-arm
        half-disk.}}\underbrace{\left(\frac{s\delta\wedge
          \eta r}{\eta r}\right)^{1-\eps}}_{\text{2-arm half
        ann.}}\underbrace{\left(\frac{r}{k\delta+r}\right)^{1-\eps}}_{\text{2-arm
      }\text{half-ann.}}\underbrace{(k\delta+r)^{1/3-\eps}}_{\text{1-arm
      }\text{half-ann`.}}.
\end{align*}
Figure~\ref{fig:conformalmap} shows how to write $\phi_{r,1}$ as a
composition of simpler conformal maps. Using this composition, we compute
\begin{align*}
\phi_{r,1}(z)&\asymp\frac{(z+r)^{2/3}}{z^{1/3}},  \\
\phi_{r,1}'(z)&\asymp\frac{z-r}{z^{4/3}(z+r)^{1/3}}, \text{ and} \\
\phi_{r,1}^{-1}(\phi_{r,1}(z))&\asymp \frac{z^{4/3}(z+r)^{1/3}}{z-r}.
\end{align*}
Using these estimates, we can upper bound $\int |\bar\partial J g_w| \,dA$
by summing over the faces $F_{n,k}$. We obtain
\begin{align*}
 \sum_{k=\eta r/\delta}^{C/\delta} \sum_{n=1}^{C k}\sum_{s=1}^{C r/\delta}
  &P(z,b)\overbrace{(k\delta+r)^{1/3}(k\delta)^{1/3}}^{(\phi^{-1})'(\phi(F_{n,k}))}
\overbrace{\delta^2
  (k\delta+r)^{-2/3}(k\delta)^{-2/3}}^{\text{area}(\phi(F_{n,k}))}
\overbrace{\frac{(k\delta+r)^{1/3}}{(k\delta)^{2/3}}}^{g_w} \\
 &\lesssim r^{-1/3} \delta^{c_3-\eps}.
\end{align*}
\end{proof}

\begin{pictures}
\begin{figure}
  \includegraphics{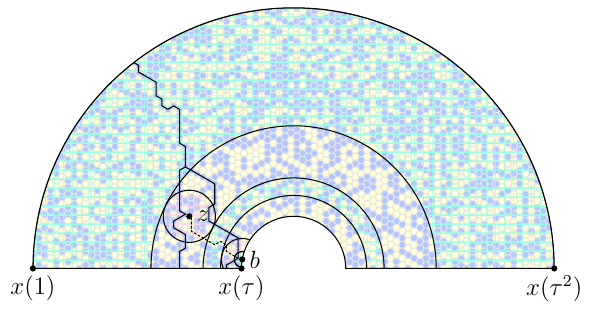}
  \caption{We use crossing events for the five regions shown to bound the
    probability of a five-arm event for which $b \in S_r$.} \label{fig:uniform_constant}
\end{figure}
\end{pictures}

\begin{halfplane}
\section{Half-plane exponent} \label{halfarmsection}

We begin with a lemma about the conformal maps $\phi_{r,R}:\Omega_{r,R}\to
T_{\text{unit}}$; see Subsection~\ref{subsec:uniform_constant} for notation.

\begin{lemma} \label{lem:confmap} There exist $a_1,a_2>0$ so that for all
  $r,R>0$ such that $r/R < 1/2$, we have
\begin{equation} \label{eq:confmap}
  a_1 \leq \frac{ \phi_{r,R}(r)}{(r/R)^{1/3}} \leq a_2.
\end{equation}
\end{lemma}

\begin{proof}
  By scaling, we may assume $R=1$. Consider the sequence of conformal maps
  illustrated in Figure~\ref{fig:conformalmap}. Let us call these maps
  $f_n$ for $n=1,2,\ldots,5$, so that $f_n:D_n\to D_{n+1}$. Since the
  domains are Jordan, we may regard $f_n$ as a continuous map defined on
  the closure of each domain. Define the compositions $\tilde{f}_n=f_n\circ
  f_{n-1}\cdots \circ f_1$.

  For $n\geq 2$, let $K_n\subset D_n$ denote the image of \[K_1\colonequals
  \{z\,:\,|z|=r\text{ and }\arg z \in [0,\pi/2]\}\cup[r,1/2]\] under
  $\tilde{f}_{n-1}$. For $n\in \{2,3,5\}$, regard $f_n$ as having been
  analytically continued in a neighborhood of every straight boundary (by
  Schwarz reflection), and define $m_n$ and $M_n$ to be the infimum and
  supremum of $f_n'(z)$ as $z$ ranges over $K_n$ and $r$ ranges over
  $[0,1/2]$.

  We claim that $0<m_n<M_n<\infty$ for all $n\in\{2,3,5\}$.  For $n=5$,
  this follows from the continuity of $f_n'$ and the fact that the
  derivative of a conformal map cannot vanish. For $n=3$, this follows from
  the joint continuity of the M\"obius map $(z-w)/(1-\overline{w}z)$ in $w$
  and $z$.

  The case $n=2$ requires more care, since the eccentricity of $D_2$
  depends on $r$. We introduce the notation $D_{2,r}$ and $f_{2,r}$ to
  indicate this dependence. Let $I\subset (0,1/2)$ be an interval. We claim
  that for every fixed $z \in \bigcap_{r\in I}D_{2,r}$, the quantity
  $f_{2,r}'(z)$ is continuous in $r$. We first recall some
    definitions from complex analysis: given a simply connected domain
    $U\subset \bbC$ and a point $z\in U$, we will say that a Riemann map
    $\varphi:\mathbb{D}\to U$ is \textit{normalized} if $\varphi(0)=z$ and
    $\varphi'(0)>0$. Recall that a sequence of open sets $U_n\subset \bbC$
    converges to an open set $U\subset \bbC$ in the Carath\'eodory sense
    with respect to $z\in U$ if (a) for all compact $K\subset U$ containing
    $z$, we have $K\subset U_n$ for all $n$ sufficiently large, and (b) $U$
    contains every open set satisfying condition (a). If $U_n \to U$ in the
    Carath\'eodory sense, then the normalized Riemann maps
    $\varphi_n:\mathbb{D} \to U_n$ converge uniformly on compact subsets to
    the normalized Riemann map $\varphi:\mathbb{D}\to U$ \cite{W}. Observe
  that if $r_n \to r$, $D_{2,r_n}$ converges to $D_{2,r}$ with respect to 0
  in the Carath\'eodory sense.  Hence $f_{2,r_n}\to f_{2,r}$ uniformly on
  compact sets, which in turn implies that $f'_{2,r_n}\to f'_{2,r}$
  uniformly on compact sets. In particular, we obtain joint
  continuity of $f_{2,r}'(z)$ in $z$ and $r$. It follows that the
  infimum and supremum of $|f_r'(z)|$ over $(z,r)\in K_n\times [0,1/2]$ are
  achieved, which implies $0<m_2<M_2<\infty$.

Since $f_1(r)=2r/(1+r^2)$, we have
\[
r \leq f_1(r) \leq 2 r.
\]
We note that each $f_n$ is monotone on the real line, and apply $f_5 \circ
f_4\circ f_3\circ f_2$ to the inequality above. Using our derivative
bounds, we obtain
% LEAVE CALCULATION COMMENTED OUT
% \[
% m_2m_3r \leq \tilde{f}_3(r) \leq 2M_2M_3 r.
% \]
% Applying $f_4(z) = z^{1/3}$, we obtain
% \[
% \left(m_2m_3r\right)^{1/3} \leq \tilde{f}_4(r) \leq \left(2M_2M_3r\right)^{1/3}.
% \]
% Finally, an application of $f_5$ gives
\[
m_5\left(m_2m_3r\right)^{1/3} \leq \tilde{f}_4(r) \leq
M_5\left(2M_2M_3r\right)^{1/3},
\]
thus the result holds with $a_1 = m_5(m_2m_3)^{1/3}$ and $a_2 =
M_5(2M_2M_3)^{1/3}$.
\end{proof}

\begin{remark}
  Numerical evidence suggests that Lemma~\ref{lem:confmap} holds with
  $a_1=1$ and $a_2\approx 1.426$.
\end{remark}

We denote by $P$ the measure $P^{\delta=1}$ corresponding to site
percolation on the triangular lattice with unit mesh size.

\begin{lemma} \label{lem:uniform} For all $0<c<c_3-1/3=1/3$ there exists $R_0 > 1$ such that
  for all $R \geq R_0$ and for all $r\leq \frac{1}{2}R$,
\begin{equation} \label{eq:uniform}
\left| P^{\delta=1}(S_r \leftrightarrow S_R) - \phi_{r,R}(r) \right|
\leq \frac{a_1}{10} R^{-c}.
\end{equation}
\end{lemma}

\begin{proof}
  This follows immediately from Proposition~\ref{prop:uniform_constant}, by
  rescaling by a factor of $R$. Note that we have used the openness of
  interval $(0,c_3-1/3)$ to deal with the multiplicative constant in the
  bound given by Proposition~\ref{prop:uniform_constant}.
\end{proof}

\begin{pictures}
\begin{figure}
  \centering
  \subfigure[\label{fig:cm1}Half-annulus $D_1$]{\includegraphics[width=4cm]{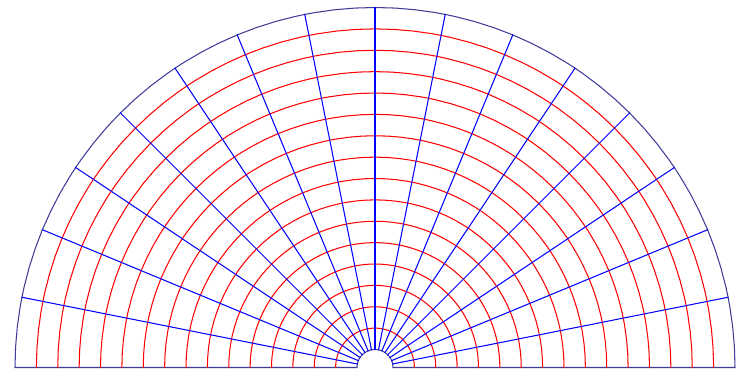}}
  \subfigure[\label{fig:cm2}Half-ellipse $D_2$]{\includegraphics[width=4cm]{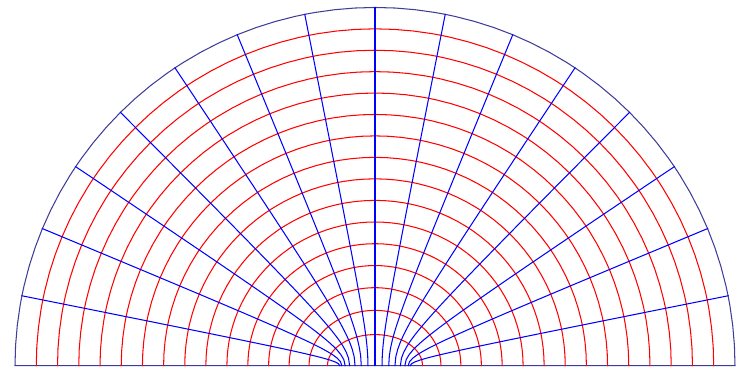}}
  \subfigure[\label{fig:cm3}Half-disk $D_3$]{\includegraphics[width=4cm]{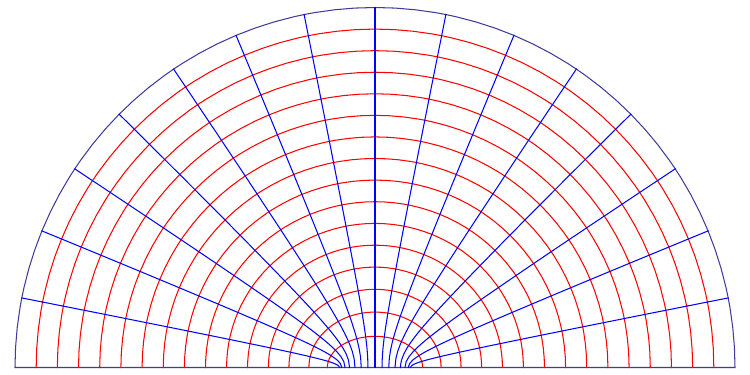}}\\
  \subfigure[\label{fig:cm4}Half-disk $D_4$]{\includegraphics[width=4cm]{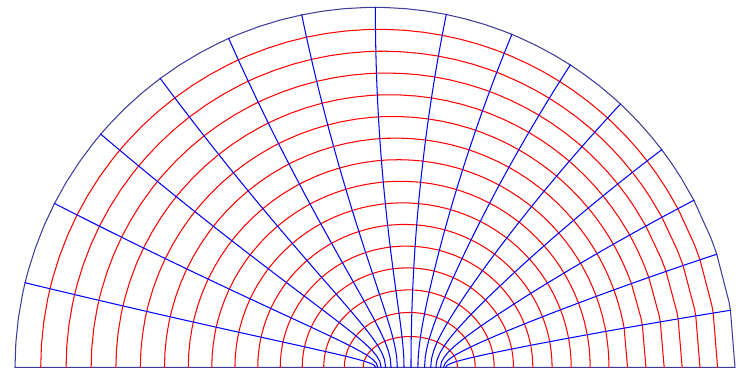}}
  \subfigure[\label{fig:cm5}Sector $D_5$]{\includegraphics[width=4cm]{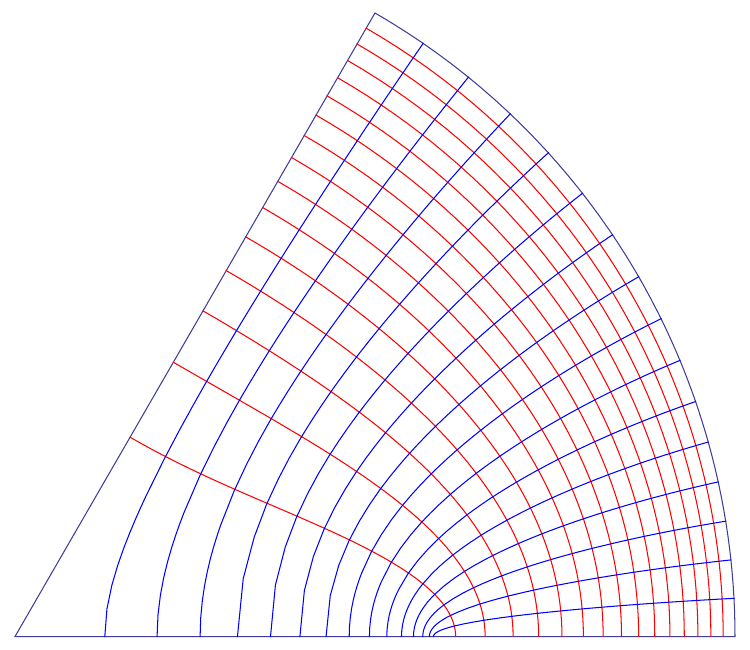}}
  \subfigure[\label{fig:cm6}Equilateral triangle $D_6$]{\includegraphics[width=4cm]{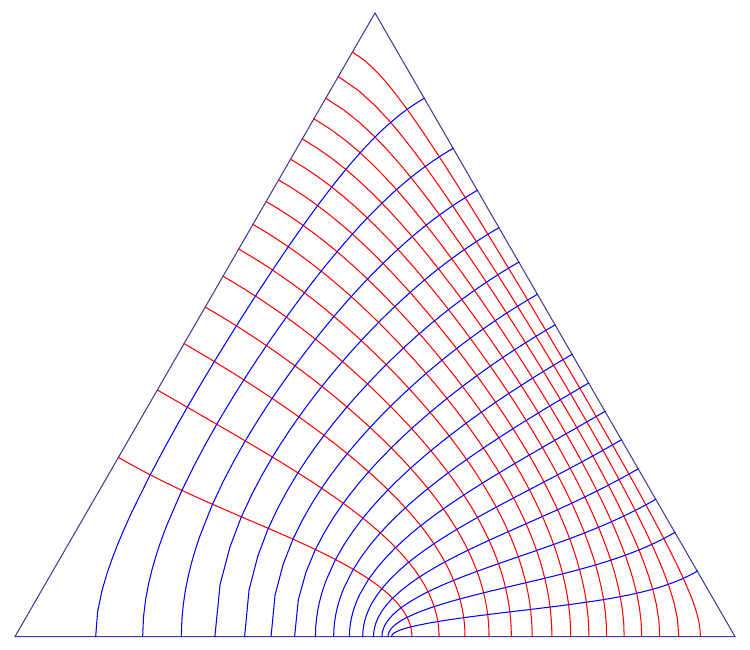}}
  \caption{Panels~(b) through (f) show the images of the half-annulus in
    panel~(a) under successive conformal maps. Composing these maps gives the
    conformal map $\phi_{r,R}$ from the half-annulus to the equilateral
    triangle which sends $-R$, $-r$, and $R$ to $e^{i\pi/3}$, 0, and 1. The
    ratio of outer radius to inner radius is 20 for the half-annulus
    shown. The map from $D_1$ to $D_2$ is a suitable scaling of $z\mapsto
    z+1/z$. The map from $D_2$ to $D_3$ is the restriction of the conformal
    map from an ellipse to the disk. From $D_3$ to $D_4$, a M\"obius map
    moves the image of $-r$ to the origin. The map from $D_4$ to $D_5$ is
    the cube root, and the map from $D_5$ to $D_6$ is the restriction to a
    sector of the Schwarz-Christoffel map from the disk to the regular
    hexagon. } \label{fig:conformalmap}
\end{figure}
\end{pictures}

\begin{proof}[Proof of Theorem~\ref{thm:halfarm}]
  Let $\eps>0$, and define $R_0=e^{\sqrt{\log \log R}}$. We assume that $R$
  is sufficiently large that $R_0$ satisfies the statement of
  Lemma~\ref{lem:uniform}.  Define $\alpha = 1/(1-3c)$ and $n = \lfloor
  \log_{R_0} \log_\alpha R \rfloor$. Let $R_k = R_0^{\alpha^k}$ for $1\leq
  k \leq n-1$, and let $R_n=R$. We first prove the upper bound.  Since an
  open path from $0$ to $S_R$ includes a crossing from $S_{R_k}$ to
  $S_{R_{k+1}}$ for all $0\leq k <n$, we may use Lemma~\ref{lem:confmap},
  Lemma~\ref{lem:uniform}, and independence to compute
\begin{align*}
P(0\leftrightarrow S_R) &\leq \prod_{k=0}^{n-1} P(S_{R_k} \leftrightarrow
S_{R_{k+1}}) \\
&\leq \prod_{k=0}^{n-1} \left[a_2\left(\frac{R_{k+1}}{R_k}\right)^{-1/3} +
  \frac{a_1}{10}R_{k+1}^{-c}\right],
\end{align*}
by (\ref{eq:uniform}). Factoring out the first term in brackets and splitting the
product, we obtain
\begin{align*}
P(0\leftrightarrow S_R) &\leq \prod_{k=0}^{n-1}a_2
\prod_{k=0}^{n-1}\left(\frac{R_{k+1}}{R_k}\right)^{-1/3}\prod_{k=0}^{n-1}
\left[1 + a_1(10a_2)^{-1}R_{k+1}^{1/3-c}R_k^{-1/3}\right] \\
&\leq (a_1/10+a_2)^{n-1} (R/R_0)^{-1/3},
\end{align*}
because the second term in brackets simplifies to $a_1/(10a_2)$ by our
choice of $R_k$. Substituting the value of $n$ gives
\begin{align*}
P(0\leftrightarrow S_R)  &\leq R_0^{1/3} (\log \alpha)^{-\log (a_1/10+a_2)/\log R_0} (\log R)^{\log (a_1/10+a_2) / \log R_0} R^{-1/3} \\
&\leq e^{C\sqrt{\log\log R}} R^{-1/3},
\end{align*}
for some constant $C$ and for sufficiently large $R$, which gives the upper
bound.

For the lower bound (see Figure~\ref{fig:lowerbound}), we define $R'_k =
2R_k$. Define $E_k$ to be the event that there is an open crossing of
$\Omega_{R_k,R_k'}$ from $[R_k,R_k']$ to $[-R_k',-R_k]$. By the
Russo-Seymour-Welsh inequality, this probability is bounded below by a
constant $p$ which does not depend on $k$. Note that there is a path from
the origin to $S_R$ if the following events occur:
\begin{enumerate}
\item there is an open path from the origin to
$S_{R'_0}$,
\item there is an open path from $S_{R_k}$ to $S_{R'_{k+1}}$ for
all $0\leq k< n $, and
\item $E_k$ occurs for all $0\leq k < n$.
\end{enumerate}

Since these events are increasing, we can use the FKG inequality to lower bound
the probability of their intersection by the product of their
probabilities. We obtain
\begin{align*}
P(0\leftrightarrow S_R) &\geq P(0\leftrightarrow S_{R_0}) \prod_{k=0}^{n-1} P(S_{R_k} \leftrightarrow
S_{R'_{k+1}}) \prod_{k=0}^{n-1} P(E_k)\\
&\geq R_0^{-1/2} p^{n-1} \prod_{k=0}^{n-1} \left[a_1\left(\frac{R_{k+1}'}{R_k}\right)^{-1/3} -
  \frac{a_1}{10}(R'_{k+1})^{-c}\right],
\end{align*}
since $P(0\leftrightarrow S_{R_0}) = R_0^{-1/3+o(1)} \gtrsim R_0^{-1/2}$,
by the Cardy-Smirnov theorem. Factoring as before and simplifying, we obtain
\begin{align*}
P(0\leftrightarrow S_R) &\geq R_0^{-1/2} (a_1p)^{n-1}
\prod_{k=0}^{n-1}\left(\frac{R'_{k+1}}{R_k}\right)^{-1/3}\prod_{k=0}^{n-1}
\left[1 - \frac{1}{10}(R'_{k+1})^{1/3-c}R_k^{-1/3}\right]\\
&\geq R_0^{-1/2} 2^{-n/3}(a_1p)^{n-1}
\prod_{k=0}^{n-1}\left(\frac{R_{k+1}}{R_k}\right)^{-1/3}\prod_{k=0}^{n-1}
\left[1 - \frac{2^{1/3-c}}{10}R_{k+1}^{1/3-c}R_k^{-1/3}\right]\\
&\geq R_0^{-1/2} 2^{-n/3}\left[a_1p(1-2^{1/3-c}/10)\right]^{n-1} (R/R_0)^{-1/3} \\
&\geq e^{-C\sqrt{\log \log R}} R^{-1/3},
\end{align*}
for some constant $C>0$ and sufficiently large $R$.
\end{proof}

\begin{pictures}
\begin{figure}
  \centering
  \includegraphics[width=12cm]{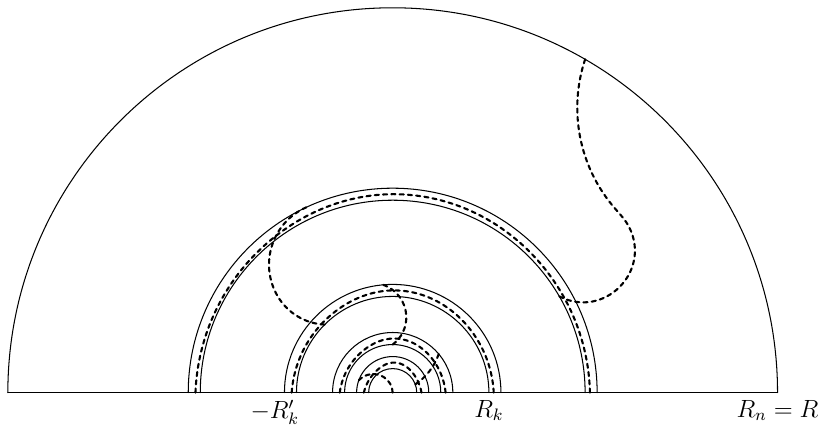}
  \caption{If there are segment-to-segment crossings of each narrow
    half-annulus, crossings from $S_{R_k}$ to $S_{R_k'}$ for each $0\leq k
    \leq n$, and an open path from the origin to $S_{R_0}$, then there is
    an open path from the origin to $S_R$. The figure shown is an image
    under radial logarithmic scaling $(r,\theta)\mapsto (\log
    r,\theta)$.} \label{fig:lowerbound}
\end{figure}
\end{pictures}

\end{halfplane}

\bigskip

\filbreak
\begingroup
\small
\parindent=0pt

\vtop{
\hsize=5.3in
\href{mailto:sswatson@math.mit.edu}{\nolinkurl{sswatson@math.mit.edu}} \\ \href{mailto:dana@math.mit.edu}{\nolinkurl{dana@math.mit.edu}} \\
Department of Mathematics\\
Massachusetts Institute of Technology \\
77 Massachusetts Ave \\
Cambridge, MA \\
USA
}

\vspace{3mm}

\vtop{
\hsize=5.3in
\href{mailto:asafnach@ubc.edu}{\nolinkurl{asafnach@math.ubc.edu}} \\
Department of Mathematics \\
University of British Columbia \\
121-1984 Mathematics Rd \\
Vancouver BC, \\
Canada
}

\enlargethispage{1cm}

\endgroup \filbreak

\end{document}